\documentclass{tac}
\usepackage[utf8]{inputenc}
\usepackage{mathtools}

\newcommand\blfootnote[1]{%
  \begingroup
  \renewcommand\thefootnote{}\footnote{#1}%
  \addtocounter{footnote}{-1}%
  \endgroup
}

\usepackage{amsmath}
\usepackage{extpfeil}
\usepackage{amssymb}
\usepackage{booktabs}
\usepackage{adjustbox}
\usepackage{ltablex,array}
\usepackage{longtable}
\usepackage{makecell}
\usepackage{marvosym}

\newcommand\restr[2]{{%
\left.
\kern-
\nulldelimiterspace %
#1 %
\right|_{#2} %
}}
\usepackage{tikz}
\usepackage{tikz-cd}
\tikzcdset{scale cd/.style={every label/.append style={scale=#1},
    cells={nodes={scale=#1}}}}
\usepackage{quiver}
\usetikzlibrary{positioning} \usetikzlibrary{decorations.pathreplacing}
\usetikzlibrary{arrows}
\tikzset{>=stealth}

\newtheorem{question}[theorem]{Question}

\usepackage[
backend=biber,
style=alphabetic,
sorting=nyt
]{biblatex}

\address{University of la R\'{e}union,
Laboratory of Mathematics and Computer Science\\
15 avenue Rene Cassin, CS92003-97744 Saint-Denis Cedex 9, France (R\'{e}union) \\[5pt]
Masaryk University, Faculty of Science,
Department of Mathematics and Statistics\\
Kotl\'{a}\v{r}sk\'{a} 2, 611 37 Brno, Czech Republic\\
}

\copyrightyear{2024}


\eaddress{christian.espindola@univ-reunion.fr \CR kanalas@mail.muni.cz}

\begin{document}
\title{Every theory is eventually of presheaf type}
\author{Christian Espíndola, Kristóf Kanalas}
\date{}

\maketitle

\begin{abstract}
    We give a detailed and self-contained introduction to the theory of $\lambda $-toposes and prove the following: 1) A $\lambda $-separable $\lambda $-topos has enough $\lambda $-points. 2) The classifying $\lambda $-topos of a $\kappa $-site $(\mathcal{C},E)$ is a presheaf topos (assuming $\kappa \vartriangleleft \lambda =\lambda ^{<\lambda }$, $|\mathcal{C}|,|E|<\lambda $). 
    \blfootnote{The second-named author acknowledges the support of the Grant Agency of the Czech Republic (grant number 22-02964S) and Masaryk University (project MUNI/A/1457/2023). This research was carried out during a visit to University of la Réunion, partially funded by France.}
\end{abstract}

\tableofcontents

\section{Introduction}

The geometric fragment $L^g_{\infty \kappa }$ of $L_{\infty \kappa }$ is the class of formulas $\forall \Vec{x}(\varphi (\vec{x}) \to \psi(\vec{x} ))$ where $\varphi $ and $\psi $ are built up from atomic formulas, $<\kappa $ $\bigwedge $ (including $\top$), arbitrary $\bigvee $ (including $\bot $) and $\exists $. This fragment (especially when $\kappa =\omega $) is the subject of categorical logic. The key connections between geometric logic and category theory are: 
\begin{itemize}
     \item[1)] Syntactic categories/ sites: The algebraization of logic using categories originated in the works of F.~W.~Lawvere (see \cite{Lawvere1} and \cite{Lawvere2}). Subsequently, the categorical structure representing theories has evolved through several stages. In these notes a theory $T\subseteq L^g_{\infty \kappa }$ will be replaced by a $\kappa $-site $(\mathcal{C}_T,E)$: a category $\mathcal{C}_T$ with $<\kappa $-limits together with a distinguished collection $E$ of arrow-families (wide cospans). 
     
     As a result $Mod(T)$, the category of $T$-models and homomorphisms will correspond to $Mod(\mathcal{C}_T,E,\kappa )$, the category of $\mathcal{C}_T\to \mathbf{Set}$ $\kappa $-lex functors turning the $E$-families into jointly epimorphic ones, see \cite[Prop.~3.2.5]{accessible} and \cite[Chapter 8]{makkai}. 
    \item[2)] Toposes: In the $\kappa =\omega $ case one can generate a Grothendieck-topology $\langle E \rangle _{\omega }$ out of $E$ and consider sheaves wrt.~it. The resulting category $Sh(\mathcal{C}_T,\langle E\rangle _{\omega })$ is the classifying topos of $T$, see \cite{elephant} and \cite{sheaves}. The generalization for $\kappa >\omega $ is discussed in Section 3.
    \item[3)] Accessible categories: For any theory $T\subseteq L_{\infty \kappa }^g$, the category $Mod(T)$ is accessible, and every accessible category arises this way, see \cite{accessible} and \cite{rosicky}. 
   
\end{itemize}

Now given a geometric theory $T\subseteq L_{\infty \omega }^g$ (equivalently: an $\omega $-site $(\mathcal{C},E)$) the following questions are hard: What is the accessibility rank of $Mod(T)$/ $Mod(\mathcal{C},E,\omega)$? Is the classifying topos/ $Sh(\mathcal{C},\langle E\rangle _{\omega })$ equivalent to a presheaf topos? (In which case $Mod(T)$ is finitely accessible, but not conversely, see \cite{thofpresh} for further details.) 

Our motto could be the following: Given any $T\subseteq L_{\infty \kappa }^g$ there is a cardinal $\lambda >\kappa $ such that the classifying $\lambda $-topos of $T\subseteq L_{\infty \kappa }^g \subseteq L_{\infty \lambda }^g$ is a presheaf topos (modulo some assumptions on cardinal arithmetic, for example GCH is more than enough).

The key ideas of this paper have their origin in \cite{inflog}, \cite{completeness} and \cite[Section 3.2 and Section 4]{eventualcat}, and our main result is a variant of \cite[Theorem 4.1]{eventualcat}. However, the discussion here is more detailed, precise and accessible. In particular our proofs are entirely category theoretic, in the sense that they avoid using techniques like "first make a theory from the site/topos in question, then change the fragment/ extend the signature/ add some axioms, finally make a site/topos from this modified theory". 

Now we comment on the structure of this paper. First the fundamental notion of compatibility between extremal epimorphic families and $<\kappa $-limits is discussed. In \cite{barrex} a $\kappa $-regular category is defined as one having $<\kappa $-limits and pullback-stable effective epi-mono factorizations, such that the transfinite cocomposition of a continuous $<\kappa $-sequence of effective epis is effective epi. The natural counterpart of this for coherent categories is formulated in terms of covering cotrees: if one builds a cotree on an object, which is locally extremal epimorphic and every branch is continuous, $<\kappa $, then the cotree is globally extremal epimorphic, i.e.~the transfinite cocomposition of the branches form an extremal epic family on the root. In the next section we study this notion in detail and prove a completeness theorem.

Then we define a $\kappa $-topos as a Grothendieck-topos whose extremal epic families are compatible with $<\kappa $-limits in the above sense. We lift the classical results concerning classifying toposes to $\kappa $-toposes. Finally we reformulate the completeness theorem in terms of $\kappa $-toposes whose defining site is small in an appropriate sense, generalizing \cite[Theorem 6.2.4]{makkai} (Theorem \ref{main1}), and derive the main result of this paper; that the classifying $\lambda $-topos of a $\kappa $-site $(\mathcal{C},E)$ is a presheaf topos, assuming $\kappa \vartriangleleft \lambda =\lambda ^{<\lambda }$, $|\mathcal{C}|,|E|<\lambda $ (Theorem \ref{main2}).

The authors want to thank Will Boney and Nate Ackerman for their valuable remarks, including a correction to the earlier version of Remark \ref{numberofbranches}. They also thank the referee for raising interesting questions, which motivated Proposition \ref{prop34}, Question \ref{quest35}, Remark \ref{rem37}, Remark \ref{rem313}, Remark \ref{rem48} and Remark \ref{rem57}.

\section{A completeness theorem}

This section motivates the definition of a $\kappa $-Grothendieck-topology, i.e.~the correct notion of compatibility between covers and $<\kappa $-limits. Its content is a completeness theorem: if a $\kappa $-lex category $\mathcal{C}$ has $\leq \kappa $-big Hom-sets then given a $\leq \kappa $-big collection $E$ of extremal epimorphic families interacting well with $<\kappa $-limits, there is a jointly conservative set of $\mathcal{C}\to \mathbf{Set}$ $\kappa $-lex $E$-preserving functors. As an application we give a Rasiowa-Sikorski-like result for sufficiently distributive lattices.
(In Section 5 we will present a more general version of the completeness theorem, based on the same idea.)

We begin with a general remark on the notation, which will also be used in later sections.

\begin{remark}
    To avoid the overuse of $*$, we will write $f^{\circ }=-\circ f$ for pre-composition with $f$, and $f_{\circ }=f\circ -$ for post-composition. Meanwhile, the inverse image of a geometric morphism $F$ will be denoted by $F^*$, and the direct image by $F_*$. Similarly, if $\varphi :\mathcal{C}\to \mathcal{D}$ is a morphism of sites or if $\varphi :\mathcal{C}\to Sh(\mathcal{D})$ is a model, then $\varphi ^*:Sh(\mathcal{C})\leftrightarrow Sh(\mathcal{D}):\varphi _*$ is the corresponding geometric morphism. 
\end{remark}

\begin{definition}
\label{compatible}
  $\kappa =cf(\kappa )\geq \aleph _0$. $\mathcal{C}$ is $\kappa $-lex. A class $E$ of extremal epimorphic families is said to be \emph{compatible with $<\kappa $-limits} if
\begin{enumerate}
    \item Each member of $E$ is pullback-stable (i.e.~its pullback along any map is extremal epimorphic). $E^{pb}$ denotes the class of extremal epimorphic families obtained this way.
    \item Given a rooted cotree (as a diagram in $\mathcal{C}$), such that on any vertex its predecessors form a member of $E^{pb}$, every branch has length $<\kappa $ and every branch is continuous (objects sitting at limit points are limits in $\mathcal{C}$); it follows that the transfinite cocomposition of the branches is extremal epimorphic.

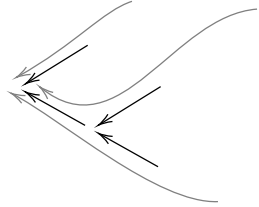
\adjustbox{scale=0.65,center}{

\tikzset{every picture/.style={line width=0.75pt}} %

\begin{tikzpicture}[x=0.75pt,y=0.75pt,yscale=-1,xscale=1]
\draw    (261,118) -- (214.7,146.94) ;
\draw [shift={(213,148)}, rotate = 327.99] [color={rgb, 255:red, 0; green, 0; blue, 0 }  ][line width=0.75]    (10.93,-3.29) .. controls (6.95,-1.4) and (3.31,-0.3) .. (0,0) .. controls (3.31,0.3) and (6.95,1.4) .. (10.93,3.29)   ;
\draw    (259,180) -- (214.74,154.98) ;
\draw [shift={(213,154)}, rotate = 29.48] [color={rgb, 255:red, 0; green, 0; blue, 0 }  ][line width=0.75]    (10.93,-3.29) .. controls (6.95,-1.4) and (3.31,-0.3) .. (0,0) .. controls (3.31,0.3) and (6.95,1.4) .. (10.93,3.29)   ;
\draw    (317,150) -- (270.7,178.94) ;
\draw [shift={(269,180)}, rotate = 327.99] [color={rgb, 255:red, 0; green, 0; blue, 0 }  ][line width=0.75]    (10.93,-3.29) .. controls (6.95,-1.4) and (3.31,-0.3) .. (0,0) .. controls (3.31,0.3) and (6.95,1.4) .. (10.93,3.29)   ;
\draw    (315,212) -- (270.74,186.98) ;
\draw [shift={(269,186)}, rotate = 29.48] [color={rgb, 255:red, 0; green, 0; blue, 0 }  ][line width=0.75]    (10.93,-3.29) .. controls (6.95,-1.4) and (3.31,-0.3) .. (0,0) .. controls (3.31,0.3) and (6.95,1.4) .. (10.93,3.29)   ;
\draw [color={rgb, 255:red, 128; green, 128; blue, 128 }  ,draw opacity=1 ]   (295,84) .. controls (273.22,90.93) and (251.44,117.46) .. (207.34,142.25) ;
\draw [shift={(206,143)}, rotate = 330.95] [color={rgb, 255:red, 128; green, 128; blue, 128 }  ,draw opacity=1 ][line width=0.75]    (10.93,-3.29) .. controls (6.95,-1.4) and (3.31,-0.3) .. (0,0) .. controls (3.31,0.3) and (6.95,1.4) .. (10.93,3.29)   ;
\draw [color={rgb, 255:red, 128; green, 128; blue, 128 }  ,draw opacity=1 ]   (391,90) .. controls (326.32,94.98) and (286.4,200.93) .. (224.93,150.77) ;
\draw [shift={(224,150)}, rotate = 39.99] [color={rgb, 255:red, 128; green, 128; blue, 128 }  ,draw opacity=1 ][line width=0.75]    (10.93,-3.29) .. controls (6.95,-1.4) and (3.31,-0.3) .. (0,0) .. controls (3.31,0.3) and (6.95,1.4) .. (10.93,3.29)   ;
\draw [color={rgb, 255:red, 128; green, 128; blue, 128 }  ,draw opacity=1 ]   (361,239) .. controls (314.71,240.97) and (240.27,175.02) .. (203.65,155.84) ;
\draw [shift={(202,155)}, rotate = 26.57] [color={rgb, 255:red, 128; green, 128; blue, 128 }  ,draw opacity=1 ][line width=0.75]    (10.93,-3.29) .. controls (6.95,-1.4) and (3.31,-0.3) .. (0,0) .. controls (3.31,0.3) and (6.95,1.4) .. (10.93,3.29)   ;

\end{tikzpicture}
}

\end{enumerate}
\end{definition}

\begin{definition}
    We shall name a few kind of diagrams: by \emph{cotree} we will mean the opposite of a rooted tree as a diagram in $\mathcal{C}$. By continuous \emph{$(E^{pb},\kappa )$-cotree} we mean that on any vertex its predecessors form a member of $E^{pb}$, every branch has length $<\kappa $ and every branch is continuous. By a continuous \emph{cofinal} $(E^{pb},\kappa ^+)$-cotree we mean that the cotree has height $\kappa $ and each branch has size $\kappa $ (hence order type $\kappa ^{op}$).
    When $E$ consists of all extremal epimorphic families with $<\lambda $-many legs, we may write locally covering continuous $(\lambda ,\kappa )$-cotree instead of continuous $(E,\kappa )$-cotree.
\end{definition}

\begin{definition}
    A category $\mathcal{C}$ is of \emph{local size} $< \kappa $ if each Hom-set is of cardinality $<\kappa $.
\end{definition}

\begin{theorem}
\label{completenesstree}
$\mathcal{C}$ is $\kappa $-lex of local size $\leq \kappa $, and it has a strict initial object $\emptyset $. Let $E$ be a set of extremal epimorphic families, such that $|E|\leq \kappa $ and $E$ is compatible with $<\kappa $-limits. Then given any object $u_0\in \mathcal{C}$ there is a continuous cofinal $(E^{pb},\kappa ^+)$-cotree with root $u_0$ such that given any branch $u_0\xleftarrow{p_{0,i_1}}u_{0,i_1}\xleftarrow{p_{0,i_1,i_2}} u_{0,i_1,i_2} \leftarrow \dots $ the colimit of the representable functors $\mathcal{C}(u_0,-)\xRightarrow{p^{\circ }}\mathcal{C}(u_{0,i_1},-)\xRightarrow{p^{\circ }} \dots $ is either the terminal copresheaf (iff the branch gets eventually $\emptyset$) or it preserves the extremal epimorphic families in $E$. When $u_0\neq \emptyset $ at least one branch yields a non-terminal copresheaf.
\end{theorem}

\begin{proof}
Let $u_0\xleftarrow{p_0}u_1\leftarrow \dots $ be an (at this point not necessarily continuous) $\kappa $-chain in $\mathcal{C}$. First we would like to understand what does it mean for the colimit of $\mathcal{C}(u_0,-)\xRightarrow{p_0^{\circ }} \mathcal{C}(u_1,-)\Rightarrow \dots $ to preserve a given extremal epimorphic family $(x_j\to y)_{j<\gamma }$ (with $\gamma \geq 1$). This is easy: it is mapped to a jointly surjective family iff for any map $u_{\alpha }\to y$ there's $u_{\beta }\to u_{\alpha }$ in the sequence such that for some $j<\gamma $ the dashed arrow 

\[\begin{tikzcd}
	{u_{\beta }} & {u_{\alpha  }} & y \\
	&& {x_0} & {x_j} & {}
	\arrow[from=2-3, to=1-3]
	\arrow[from=2-4, to=1-3]
	\arrow[from=1-1, to=1-2]
	\arrow[from=1-2, to=1-3]
	\arrow[curve={height=40pt}, dashed, from=1-1, to=2-4]
	\arrow["\dots"{description}, draw=none, from=2-3, to=2-4]
	\arrow["\dots"{description}, draw=none, from=2-4, to=2-5]
\end{tikzcd}\]
exists.

We shall also understand the $\gamma =0$ case. The empty collection of arrows is extremal epi on $y$ iff $y$ has no proper subobjects. As the initial object is strict, every map out of it is mono, hence this is the same as $y\cong \emptyset $. The colimit of the representables does not preserve the initial object iff one of the $u_{\alpha }$'s is initial, in which case all the following $u_{\beta }$'s are initial and hence the colimit is the terminal copresheaf which is allowed. So we assume that the prescribed families are non-empty.

One such issue can certainly be solved at any stage of the construction: if we have started to construct the chain (and we are in a successor step) $u_0\leftarrow \dots u_{\beta }$, then given $\alpha \leq \beta $ and a map $u_{\alpha }\to y\leftarrow x_{j}$ just take the pullback
\[\begin{tikzcd}
	{u_{\beta +1}} && {x_{j}} \\
	{u_{\beta }} & {u_{\alpha }} & y
	\arrow[from=1-3, to=2-3]
	\arrow[from=2-1, to=2-2]
	\arrow[from=2-2, to=2-3]
	\arrow[from=1-1, to=2-1]
	\arrow[""{name=0, anchor=center, inner sep=0}, from=1-1, to=1-3]
	\arrow["pb"{description, pos=0.6}, draw=none, from=0, to=2-2]
\end{tikzcd}\]
(for any $j$: hence we have $\gamma $-many options to continue the chain, these possibilities will form the cotree that we promised).

As our Hom-sets have size $\leq \kappa $ and we have $\kappa $-many prescribed families, there are $\kappa $-many tasks concerning a given $u\in \mathcal{C}$. But solving one such issue will yield a new $u_{\alpha }$ and hence $\kappa $-many new tasks. Luckily we have $\kappa $-many steps to arrange everything and this can be done:

For any $u\in \mathcal{C}$ let $T_u$ be the set (or list) of diagrams $u\to y\leftarrow x_{0}, x_1, \dots $ (where $(y\leftarrow x_j)_j$ is any member of $E$), well-ordered in order type $\kappa $ (technicality: assume $1\to 1$ is among the given families, so $T_u$ is non-empty. Allowing repetitions we can assume that $T_u$ has size $\kappa $). We have the canonical well-ordering $f:\kappa \times \kappa \to \kappa $ with the property $f(\alpha ,\beta )\geq \beta $. We use it as follows:

Let's fix a table of size $\kappa \times \kappa $ (the second coordinate labels the columns) whose entries are empty at the moment. As the root of the cotree is $u_0$ we list $T_{u_0}$ in the zeroth column: these tasks will have to be solved. Now, at step $1$ we solve $f^{-1}(0)$, whose second coordinate is $\leq 0$, so it is one of the already listed tasks. We do it by forming the pullback of the covering family; this yields the first level of the cotree

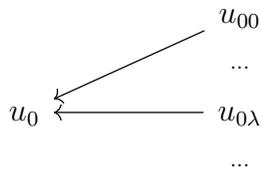
\[\begin{tikzcd}
	&& {u_{00}} \\
	{u_0} && {u_{0\lambda  }} \\
	&& {}
	\arrow[from=1-3, to=2-1]
	\arrow[from=2-3, to=2-1]
	\arrow["\dots"{description}, draw=none, from=1-3, to=2-3]
	\arrow["\dots"{description}, draw=none, from=2-3, to=3-3]
\end{tikzcd}\]
To construct the continuation from $u_{0\lambda }$ fill in $T_{u_{0,\lambda }}$ in the first column and solve $f^{-1}(1)$, whose second coordinate is $\leq 1$ hence we know which task it is.

The transfinite recursion is then the obvious thing: to define the $\alpha $-level of the cotree at limit $\alpha $ just take the transfinite cocomposition of the branches. Now do it for $\alpha +1$: to continue from a given vertex $u_{\dots }$ pick the corresponding table whose columns before $\alpha $ are filled in (with the tasks corresponding to the elements which are above $u_{\dots}$). Now fill $T_{u_{\dots }} $ into the $\alpha ^{\text{th}}$ column and solve task number $f^{-1}(\alpha )$.

This defines a continuous cofinal $(E^{pb} ,\kappa ^+)$-cotree. When we go along a branch we see that every task is solved, so the colimit preserves all the given extremal epimorphic families, except the empty union which we cancelled from the list. So every branch yields either the terminal copresheaf or an $E$-preserving one.

It remains to prove that when $u_0\neq \emptyset $ at least one branch yields a non-terminal copresheaf, i.e.~that it cannot happen that every branch is becoming constant $\emptyset $ at some point. But this is clear, otherwise (by cutting down each branch at that point) we would get a continuous $(E^{pb} ,\kappa )$-cotree which is not covering the root. 
\end{proof}

From this we derive the completeness theorem:

\begin{theorem}
\label{setcompleteness}
    Let $\mathcal{C}$ be $\kappa $-lex of local size $\leq \kappa $ with a strict initial object, and let $E$ be a set of extremal epimorphic families, such that it is compatible with $<\kappa $-limits and $|E|\leq \kappa $. Then given $x\in \mathcal{C}$ and $u,v$ subobjects of $x$, if for every $M:\mathcal{C}\to \mathbf{Set}$ $\kappa $-lex $E$-preserving functor we have $Mu \subseteq Mv$ then $u\subseteq v$.
\end{theorem}

\begin{proof}
    Take the above cotree with root $u=u_0$. If a branch $h$ yields a $\kappa $-lex $E$-preserving functor $M_h=colim (\mathcal{C}(u,-)\xRightarrow{p^{\circ }}\mathcal{C}(u_{0,h(1)},-)\Rightarrow \dots )$ then by assumption $M_h(u)\subseteq M_h(v)$, in particular $1_u\in M_{h}(u)\subseteq M_{h}(x)$ lies in $M_h(v)$, meaning that for some $u_{\lambda}$ in the branch we have a lift
\[\begin{tikzcd}
	x & v \\
	u && {u_{\lambda}}
	\arrow[hook, from=2-1, to=1-1]
	\arrow[hook', from=1-2, to=1-1]
	\arrow[from=2-3, to=2-1]
	\arrow[dashed, from=2-3, to=1-2]
\end{tikzcd}\]
When the branch becomes eventually $\emptyset $ we also have this lifting. Cutting down each branch at such a point and using that the resulting $(E^{pb} ,\kappa )$-cotree is covering the root we have an extremal epimorphic family with liftings 

\[\begin{tikzcd}
	x & v \\
	u && {u_0} \\
	{} & {u_{\lambda }}
	\arrow[from=2-3, to=2-1]
	\arrow["\dots"{description}, draw=none, from=2-3, to=3-2]
	\arrow[from=3-2, to=2-1]
	\arrow["\dots"{description}, draw=none, from=3-2, to=3-1]
	\arrow[hook, from=2-1, to=1-1]
	\arrow[hook', from=1-2, to=1-1]
	\arrow[dashed, from=2-3, to=1-2]
	\arrow[dashed, from=3-2, to=1-2]
\end{tikzcd}\]
Hence each $u_{\lambda }\to u$ factors through $v\cap u$ and therefore $v\cap u =u$ so $u\subseteq v$.
\end{proof}

\begin{definition}
$\lambda =cf(\lambda )\geq \kappa =cf(\kappa )\geq \aleph _0$. A category $\mathcal{C}$ is \emph{$(\lambda ,\kappa )$-coherent} if
    \begin{itemize}
        \item[i)] it has $<\kappa $-limits,
        \item[ii)] it has pullback-stable effective epi - mono factorization,
        \item[iii)] it has pullback-stable $<\lambda $-unions,
        \item[iv)] and $<\lambda $ extremal epimorphic families are compatible with $<\kappa $-limits (pullback-stability follows from ii) and iii), so the additional requirement is the second clause of Definition \ref{compatible}: that locally covering continuous $(\lambda ,\kappa )$-cotrees are globally covering).
    \end{itemize}

A functor is \emph{$(\lambda ,\kappa )$-coherent} if it preserves $<\kappa $-limits, effective epimorphisms and $<\lambda $-unions.
\end{definition}

\begin{remark}
\label{numberofbranches}
We shall compute the number of branches of the tree in $iv)$.

Recall, that an infinite cardinal $\kappa $ has the tree property if given a tree of height $\kappa $ such that every level is of size $<\kappa $, it follows that the tree has a cofinal branch.

Since $\lambda $ is regular there are $<\lambda $ objects in each level below $\mu $, assuming $(<\lambda )^{<\mu }=(<\lambda )$. Since every branch is $<\kappa $ the tree has height $\leq \kappa $. Therefore if $\lambda >\kappa $ with $(<\lambda )^{\kappa }=(<\lambda )$ or if $\lambda =\kappa $ has the tree property and satisfies $(<\kappa )^{<\kappa }=(<\kappa )$, it follows that the tree has $<\lambda $ branches. The latter case is equivalent to $\kappa $ being weakly compact.
\end{remark}

\begin{remark}
    When $\kappa =\aleph _0$ iv) is redundant, i.e.~$(\aleph _0,\aleph _0)$-coherent is just coherent, $(\lambda ,\aleph _0)$-coherent is $\lambda $-geometric. Indeed, given a cotree as in iv) with root $u_0$ and an arbitrary proper subobject $v\subsetneqq u_0$ one of the predecessors $u_1$ of the root does not factor through $v$. That is, the pullback of $v$ along $u_0\leftarrow u_1$ is a proper subobject $v'\subsetneqq u_1$.
    Hence one of the predecessors of $u_1$ does not factor through $v'$, equivalently: through $v$. Since we cannot define an infinite branch we get stuck at a finite stage, meaning that the (co)composition of a branch does not factor through $v$.
\end{remark}

\begin{example}
\label{preshkappatopos}
    $\mathbf{Set}$ is $(\lambda ,\kappa )$-coherent for any $\lambda ,\kappa $. Given a set $X$ and a locally covering continuous cotree on it, we shall see that the transfinite cocomposition of the branches cover $X$. But any element $x_0\in X$ has a preimage in one of the predecessors $X_1$. Once we define a compatible family of preimages up to height $\alpha $, a limit ordinal, this sequence represents an element in the limit, so we've managed to find a preimage $x_{\alpha }=(x_i)_{i<\alpha } \in X_{\alpha }=lim _{i<\alpha } X_i$. This defines a branch whose cocomposition hits $x_0$.

    As in presheaf categories limits and colimits are pointwise, also $\mathbf{Set}^{\mathcal{A}}$ is $(\lambda ,\kappa )$-coherent for any $\lambda ,\kappa$.
\end{example}

\begin{definition}
    A $\lambda $-complete Boolean-algebra is \emph{$(2,\lambda)$-distributive} if for any collection of elements $(b_{i,0})_{i\in I}$, $(b_{i,1})_{i\in I}$ with $|I|<\lambda $, we have $\bigcap _i (b_{i,0} \cup b_{i,1})=\bigcup _{h:I\to 2} \bigcap _i b_{i,h(i)}$. In particular the union exists, though we do not require the existence of such big unions in general.
\end{definition}

\begin{example}
    ($\lambda \geq \kappa $ are infinite regular cardinals.) Every $\lambda $-complete $(2,\lambda )$-distribu\-tive Boolean-algebra is $(\lambda ,\kappa )$-coherent if $\kappa <\lambda $ and  $(<\lambda )^{< \kappa }=(<\lambda )$ or if $\kappa = \lambda$ is weakly compact. 
    
    Given a continuous, locally covering $(\lambda ,\kappa )$-cotree of elements
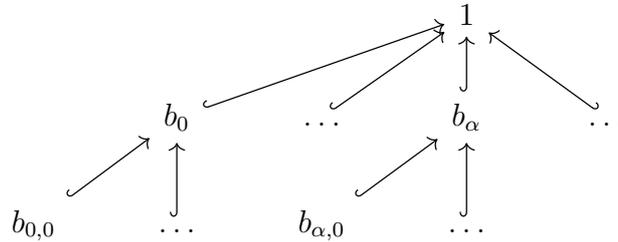
\[\begin{tikzcd}
	&&& 1 \\
	& {b_0} & {\dots } & {b_{\alpha }} & \dots \\
	{b_{0,0 }} & \dots & {b_{\alpha ,0}} & \dots
	\arrow[hook, from=2-2, to=1-4]
	\arrow[hook, from=2-4, to=1-4]
	\arrow[hook, from=3-1, to=2-2]
	\arrow[hook, from=3-2, to=2-2]
	\arrow[hook, from=2-3, to=1-4]
	\arrow[hook', from=2-5, to=1-4]
	\arrow[hook, from=3-3, to=2-4]
	\arrow[hook, from=3-4, to=2-4]
\end{tikzcd}\]
we shall see that intersections along the branches cover the root (which can be assumed to be $1$, otherwise just put its complement next to it and put $1$ to the $-1^{\text{th}}$ level). This means that for any $0\neq u\subseteq 1$ there's a branch whose intersection has non-empty intersection with $u$. 

By the cardinality assumptions there are $<\lambda $-many elements in the tree (there are $<\lambda $ elements in each level below $\kappa $, when $\kappa <\lambda $ their sum is $<\lambda $ by regularity, and when $\kappa =\lambda $ is weakly compact the tree has height $<\kappa $, so again, the sum is $<\kappa $). Let $A$ be the set of these elements. By $(2,\lambda )$-distributivity $1=\bigcap _{b\in A} (b\cup \neg b)=\bigcup _{\varepsilon :A\to \{+,-\}} \bigcap _{b\in A} b^{\varepsilon (b)}$. As $u\neq 0$ it has non-empty intersection with a summand (otherwise $\neg u$ would be a smaller upper bound). Let $u'$ be this intersection. Then it is atomic in the sense that for any $b\in A$ either $u'\subseteq b$ or $u'\subseteq \neg b$.

A member of the union is coming from decorating the tree with signs and then taking intersections of the elements or their complements (depending on the sign). If a decorated tree has non-empty intersection with $u'$ then $1$ has decoration $+$, given any element with decoration $+$ one of the predecessors is decorated with $+$ and given any chain of $+$-coloured elements the intersection must be $+$-coloured as well. So $u$ has non-empty intersection with the intersection of a decorated tree that has a positive branch.  
\end{example}

\begin{corollary}[of Theorem \ref{setcompleteness}]
    Let $\mathcal{L}$ be a distributive lattice which is a $(\lambda ,\kappa )$-coherent category and let $(A_i)_{i<\kappa }$ be a collection of subsets each of size $<\lambda $. Then there's an injective homomorphism of posets $\mathcal{L}\hookrightarrow \mathcal{P}(X)$ to a power set Boolean-algebra, which preserves all $<\kappa $-meets and which preserves $\bigcup A_i$ for each $i$.
\end{corollary}

\begin{proof}
    If $a\neq b$ then either $a\cap b\subsetneqq b$ or $a\cap b \subsetneqq a$, hence they can be separated by an $\mathcal{L}\to \mathbf{Set}$ functor which preserves all $<\kappa $-meets, monos, the terminal object, as well as the prescribed unions. In particular it lands in $2\hookrightarrow \mathbf{Set}$. Putting these together yields the required homomorphism $\mathcal{L}\to 2^X$.
\end{proof}

\section{Classifying toposes}

\begin{definition}
    A category $\mathcal{E}$ is a \emph{$\kappa $-topos} if it is a Grothendieck-topos which is $(\infty ,\kappa )$-coherent. A geometric morphism $F_*:\mathcal{E}_1\to  \mathcal{E}_2$ is a map of $\kappa $-toposes if $F^*$ preserves $<\kappa $-limits. 
\end{definition}

\begin{remark}
    By the adjoint functor theorem a map of $\kappa $-toposes $\mathcal{E}_1\to \mathcal{E}_2$ is the same as an $\mathcal{E}_2\to \mathcal{E}_1$ functor preserving $<\kappa $-limits and all colimits. We write $Fun^*_{\kappa }(\mathcal{E}_2,\mathcal{E}_1)$ for the category of $\kappa $-lex cocontinuous functors and all natural transformations. (It is locally small.)
\end{remark}

\begin{example}
    By Example \ref{preshkappatopos} every presheaf topos is a $\kappa $-topos for any $\kappa $.
\end{example}

The next proposition is a partial converse to this:

\begin{proposition}
\label{prop34}
    Assume that $\mathcal{E}$ is a $\kappa $-topos for all $\kappa $ and the subobjects of $1$ form a generating set. Then $\mathcal{E}$ is a presheaf topos. 
\end{proposition}

\begin{proof}
    By the comparison lemma (\cite[Theorem C2.2.3]{elephant}) $\mathcal{E}\simeq Sh(Sub_{\mathcal{E}}(1))$. By assumption $Sub_{\mathcal{E}}(1)$ is an $(\infty ,\infty )$-coherent distributive lattice. Applying Theorem \ref{completenesstree} we get that every element of $Sub_{\mathcal{E}}(1)$ can be written as a union of join-irreducible elements. Let $X$ be the poset of join-irreducible elements in $Sub_{\mathcal{E}}(1)$. Applying the comparison lemma again, $\mathcal{E}$ is equivalent to sheaves on $X$ with the restricted union-topology. But that is the trivial topology and hence $\mathcal{E}\simeq \mathbf{Set}^{X^{op}}$. 
\end{proof}

\begin{question}
\label{quest35}
    Is it true that if $\mathcal{E}$ is a $\kappa $-topos for all $\kappa $ then $\mathcal{E}$ is a presheaf topos?
\end{question}

What we call a $\kappa $-topos here is the same as a $\kappa $-geometric topos from \cite{inflog}. In \cite{henry} the term "$\kappa $-topos" is used for $\kappa $-lex localizations of presheaf toposes. The following example connects these notions:

\begin{example}
    While every $\kappa $-topos (in our sense) is the $\kappa $-lex localization of a presheaf topos (see Proposition \ref{sheavesonksite}.(4)), the converse is false. (This claim appears in \cite{henry}, based on \cite[Remark 1.1.4]{inflog}.) 
    
    Consider the site $([0,1],\tau _{sup})$ where $[0,1]$ is seen as a poset with the usual ordering and $(r_i\leq x)_{i\in I}$ is a cover if $sup\{ r_i :i\in I \} = x $. $[0,1]$ has all limits (meets) given by infimum, and $\tau _{sup}$ is a Grothendieck-topology.

    We claim that the sheafification map $a:\mathbf{Set}^{[0,1]^{op}}\to Sh([0,1],\tau _{sup})$ preserves all limits. This follows as on any object $r\in [0,1]$ there are only two covering sieves: $[0,r)$ and $[0,r]$, hence the $+$-construction is given by a limit formula (see \cite[p.134]{sheaves}).

    However, $Sh([0,1],\tau _{sup})$ is not an $\aleph _1$-topos. We build a locally covering continuous cotree (inside $[0,1]$) on $1$ as follows: the root is $1$, given an object $r$ in the cotree, if $r=0$ then it has no predecessors, if $r>0$ then its predecessors are given by an $\omega $-sequence converging to $r$, formed by elements that are $<r$. At limit stages we take infimums. This defines a locally covering (the predecessors of any object form a $\tau _{sup}$-cover), continuous cotree such that every branch is countable (there is no uncountable strictly decreasing sequence in $[0,1]$), and every branch terminates in $0$. By the previous paragraph $[0,1]\xrightarrow{y}\mathbf{Set}^{[0,1]^{op}}\xrightarrow{a} Sh([0,1],\tau _{sup })$ takes infimums to limits ($y$ preserves all limits), it maps $\tau _{sup}$-covers to extremal epimorphic families, and in particular takes $0$ to the initial and $1$ to the terminal object. So if $Sh([0,1],\tau _{sup})$ is an $\aleph _1$-topos, then the branches of the $ay$-image of our tree form an extremal epimorphic family, therefore $\emptyset =*$. This is not the case as $\tau _{sup}$-covers are effective epimorphic, hence $ay$ is conservative (the representables are sheaves).
 \end{example}

\begin{remark}
\label{rem37}
    It is clear that every $\kappa $-topos is a $\kappa $-regular category (the transfinite cocomposition of a continuous $<\kappa $ chain of effective epis is an effective epi). The converse is false: we claim that $Sh([0,1],\tau _{\sup })$ is an $\aleph _1$-regular topos.

    First note that it is enough to consider chains indexed by $\omega ^{op}$. Indeed, if $\alpha $ is any countable ordinal then there is a cofinal map $\omega \hookrightarrow \alpha $, and since cones on $\alpha ^{op}$ are in bijective correspondence with cones on $\omega ^{op}$, taking the limit along this subsequence is the same as taking the limit of the original chain.

    Recall that in any Grothendieck-topos $Sh(\mathcal{C})$, a family $(F\Leftarrow F_i)_{i\in I}$ is (effective) epimorphic iff the following is satisfied: for any $x\in \mathcal{C}$ and $s\in F(x)$ there's a cover $(x\leftarrow x_j)_{j\in J}$ such that each $\restr{s}{x_j}$ is the image of some $s'\in F_{i(j)}(x_j)$.

    So let $F_0\xtwoheadleftarrow{f_0} F_1\xtwoheadleftarrow{f_1} \dots $ be an $\omega ^{op}$-chain of effective epis, write $(f_{\omega ,i}:F_{\omega }\to F_i)_i$ for the limit cone and take $s\in F_0(p_0)$ for some $p_0\in [0,1]$. We have to show that $p_0$ admits a $\tau _{sup}$-cover, such that the restrictions of $s$ have preimages along the respective components of $f_{\omega ,0}$.

    Since $f_0$ is epi we get a cover $(p_{0,i_1}\leq p_0)_{i_1\in I_1}$ such that $\restr{s}{p_{0,i_1}}$ has a preimage along $(f_{0})_{p_{0,i_1}}$. For fixed $i_1$ let $s^{0,i_1}\in F_1(p_{0,i_1})$ be such a preimage. By iterating this we obtain a covering tree of height $\omega $. Given any branch $\vec{i}=0,i_1,i_2,\dots $ write $p_{\vec{i}}=\bigcap \{p_0,p_{0,i_1},\dots  \}$. Then the compatible family $(\restr{s}{p_{\vec{i}}},\restr{s^{0,i_1}}{p_{\vec{i}}}, \dots )$ gives a preimage of $\restr{s}{p_{\vec{i}}}$ along $(f_{\omega ,0})_{p_{\vec{i}}}$.

    So it is enough to show that given a locally $\tau _{sup}$-covering tree of height $\omega $ in $[0,1]$ with root $p$, the intersections of the branches can get arbitrarily close to $p$. This is easy. For any fixed $\varepsilon >0$ there is an element $p_1$ in the first level of the tree that is strictly above $p-\varepsilon $. Similarly one of the predecessors of $p_1$ is strictly above $p-\varepsilon $. The infimum of the elements in the resulting branch $\bigcap \{p,p_1,p_2,\dots \}$ is $\geq p-\varepsilon $.  
\end{remark}

\begin{definition}
    We fix the notation: Let $\mathcal{C}$ be a small $\kappa $-lex category and $E$ be a set of arbitrary families in $\mathcal{C}$ (a family is a set of arrows with common codomain). By $E^{pb}$ we denote the closure of $E$ under pullbacks; those families which can be obtained as a pullback of an $E$-family. $E^{tree_{\kappa }}$ is the closure of $E$ under the $\kappa $-tree operation: the set of those families which are obtained as transfinite cocompositions of continuous $(E,\kappa )$-cotrees (the predecessors on any node form an $E$-family, every branch is continuous, $<\kappa $). We write $\langle E\rangle _{\kappa }$ for $(E^{pb})^{tree _{\kappa }}$.

    A (small) \emph{$\kappa $-site} is a pair $(\mathcal{C},E)$ where $\mathcal{C}$ is $\kappa $-lex (small) and $E$ is an arbitrary collection of families containing $id:1\to 1$.
    $F:(\mathcal{C},E)\to (\mathcal{D},E')$ is a \emph{morphism of $\kappa $-sites } if $F$ is a $\kappa $-lex $\mathcal{C}\to \mathcal{D}$ functor which takes $E$-families (and hence $\langle E\rangle _{\kappa }$-families) to $\langle E'\rangle _{\kappa }$-families.
\end{definition}

\begin{remark}
    What we call a $\kappa $-site is in fact a sketch. As such, $\kappa $-sites correspond to theories in $L_{\infty \kappa }^g$, see \cite[Theorem 3.2.1, Proposition 3.2.5]{accessible}.
\end{remark}

\begin{remark}
    Usually a site is a category $\mathcal{C}$ equipped with a Grothendieck-topology (and sometimes $\mathcal{C}$ is assumed to be lex). So it would make sense to define a $\kappa $-site as a $\kappa $-lex category equipped with a $\kappa $-topology (that is: a Grothendieck-topology which is closed under the $\kappa $-tree operation). But since in practice topologies often arise from a generating set of families, and since in most of our results the size of this generating set matters, we found it notation-wise more convenient to use the definition given above.
\end{remark}

\begin{proposition}
\label{sheavesonksite}
    Let $(\mathcal{C},E)$ be a small $\kappa $-site. Then:
    \begin{enumerate}
        \item $\langle E\rangle _{\kappa }$ is closed under pullbacks and the $\kappa $-tree operation.
        \item $\langle E\rangle _{\kappa }$ is a Grothendieck-topology.
        \item $Sh(\mathcal{C}, \langle E\rangle _{\kappa })$ is a $\kappa $-topos.
        \item The full inclusion $Sh(\mathcal{C},\langle E\rangle _{\kappa } )\hookrightarrow \mathbf{Set}^{\mathcal{C}^{op}}$ is a map of $\kappa $-toposes, i.e.~the sheafification functor $a$ is $\kappa $-lex.
        \item The Yoneda-embedding followed by sheafification $ay:\mathcal{C}\xrightarrow{y} Set^{\mathcal{C}^{op}}\xrightarrow{a} Sh(\mathcal{C},\langle E\rangle _{\kappa } )$ is $\kappa $-lex, $E$-preserving (maps $E$-families to extremal epimorphic ones).
    \end{enumerate}
\end{proposition}

\begin{proof}
    \begin{enumerate}
        \item The pullback of a continuous $(E^{pb},\kappa )$-cotree is again a continuous $(E^{pb},\kappa )$-cotree, as the pullback of a pulled back $E$-family is the pullback of the original $E$-family along the composite and the pullback of a transfinite cocomposition is the same as the transfinite cocomposition of the pullbacks. Closed under the tree operation: the cotree we build can be seen as pasting some $(E^{pb},\kappa)$-cotrees together, which will be a continuous $(E^{pb},\kappa )$-cotree by the regularity of $\kappa $, and as the (transfinite) cocomposition of the transfinite cocompositions is a transfinite cocomposition.
        \item It contains the isomorphisms as those are pullbacks of $id:1\to 1$. The rest of the requirements (being closed under pullbacks and being closed under the tree operation when one is building cotrees of height 2) is implied by Claim 1.
        \item The argument will be similar to the one in Remark \ref{rem37}.
        
        Take a locally covering continuous $(\infty ,\kappa )$-cotree (i.e.~every node is covered by an arbitrarily large extremal epimorphic family, each branch is $<\kappa $ and continuous), let $F_0$ be the root and take an $x_0\in \mathcal{C}$ and $s^{0}\in F_0(x_0)$. Then $x_0$ has a cover $(x_0\leftarrow x_{0,j_1})_{j_1}$, such that the restrictions of $s^0$ are coming from some $F_i$'s (there can be many, so fix suitable $i(0,j_{1})$'s and preimages $s^{0,j_1}\in F_{i(0,j_1)}(x_{0,j_1})$). Now each $x_{0,j_1}$ has a cover $(x_{0,j_1}\leftarrow x_{0,j_1,j_2})_{j_2}$ such that the restrictions of $s^{0,j_1}$'s are coming from some predecessors of $F_{i(0,j_1)}$. So we started to build a locally covering continuous $((E^{pb})^{tree},\kappa )$-cotree on $x_0$, together with a morphism $i$ of cotrees, from the cotree on $x_0$ to the one on $F_0$ (meaning: $i(0,j_1,\dots )$ is an initial segment of $i(0,j_1,\dots ,j_k)$). It has the property that $\restr{s^0}{x_{\vec{j}}}$ has a preimage in $F_{i(\vec{j})}$.
        
        Once we are at a limit stage, for any branch $x_0\leftarrow x_{0,j_1}\leftarrow x_{0,j_1,j_2}\leftarrow \dots $ its limit $x_{0,j_1,\dots }$ has the property that $\restr{s^0}{x_{0,j_1,\dots }}$ has a preimage in $F_{i(0,j_1,\dots ) }(x_{0,j_1,\dots})$, namely the compatible family $[\restr{s^0}{x_{0,j_1,\dots }}, \restr{s^{0,j_1}}{x_{0,j_1,\dots }},\dots ]$ (which is in the limit of $F_0\Leftarrow F_{i(0,j_1)}\Leftarrow \dots $ because limits of sheaves are pointwise).

        So we managed to define a locally covering continuous $((E^{pb})^{tree},\kappa )$-cotree (every branch has length $<\kappa $ because the same holds in the cotree of sheaves and $i$ is a map of cotrees), and the restriction of $s^0$ to the transfinite cocomposition of a branch $x_{\vec{j}}$ has a preimage in the transfinite cocomposition of $F_{i(\vec{j})}$. This family on $x_0$ is a cover by Claim 1. 
        \item The poset of coveres on an object $x\in \mathcal{C}$ (ordered by refinement) is $\kappa $-filtered (one can build a tree whose first level consists of the $0^{\text{th}}$ family and whose $\alpha +1^{\text{th}}$ level is formed by pulling back the $\alpha ^{\text{th}}$ family to the leafs). Hence the $+$-construction is $\kappa $-lex.
        \item The composite is $\kappa $-lex by Claim 4. It is true for any $\omega $-site that $ay$ takes covering families to epimorphic ones, see e.g.~\cite[Corollary III.7.7]{sheaves}. 
    \end{enumerate}
\end{proof}

\begin{corollary}
    Let $(\mathcal{C},E)$ and $(\mathcal{D},E')$ be small $\kappa $-sites and $F:(\mathcal{C},E)\to (\mathcal{D},E')$ be a morphism of $\kappa $-sites. Then there is an induced morphism of $\kappa $-toposes $F_*:Sh(\mathcal{D},\langle E'\rangle _{\kappa })\to Sh(\mathcal{C},\langle E\rangle _{\kappa })$ where $F_*$ is $-\circ F^{op}$ and $F^*$ is $Sh(\mathcal{C})\hookrightarrow \mathbf{Set}^{\mathcal{C}^{op}}\xrightarrow{Lan_{F^{op}}}\mathbf{Set}^{\mathcal{D}^{op}}\xrightarrow{a}Sh(\mathcal{D})$.
\end{corollary}

\begin{proof}
    The geometric morphism exists by \cite[Corollary C2.3.4]{elephant}, $Lan_{F^{op}}$ is $\kappa $-lex by \cite[Example A4.1.10]{elephant} and $a$ is $\kappa $-lex by Claim 4.
\end{proof}

\begin{remark}
\label{rem313}
    Every $\kappa $-topos is the category of sheaves on a $\kappa $-site and every morphism of $\kappa $-toposes is induced by a morphism of $\kappa $-sites. Indeed, let $\mathcal{E}$ be a $\kappa $-topos and let $\mathcal{C}$ be any small full subcategory which contains a set of generators and which is closed under $<\kappa $-limits. By \cite[Theorem C2.2.3]{elephant} (the inverse image part of) the induced geometric morphism $Sh(\mathcal{C},\tau _{can})\to \mathcal{E}$ is an equivalence (here $\tau _{can}$ is the set of effective epimorphic families). Since $\mathcal{E}$ is a $\kappa $-topos, its canonical topology is a $\kappa $-topology, hence $\tau _{can}=\langle \tau _{can}\rangle _{\kappa }$. The other part of the claim is proved similarly: if $F^*:\mathcal{E}\to \mathcal{E}'$ is a map of $\kappa $-toposes and $\mathcal{E}\simeq Sh(\mathcal{C},\tau _{can})$ as before, then let $\mathcal{D}$ be a small full subcategory of $\mathcal{E}'$ which is closed under $<\kappa $ limits, contains a set of generators and which contains the image of $\restr{F^*}{\mathcal{C}}$. Then $\restr{F^*}{\mathcal{C}}: (\mathcal{C},\tau _{can})\to (\mathcal{D},\tau _{can})$ is a morphism of $\kappa $-sites, and the induced geometric morphism is $F$.
\end{remark}

\begin{definition}
     Given a small $\kappa $-site $(\mathcal{C},E)$, we write $Mod(\mathcal{C})=Mod(\mathcal{C},E,\kappa )$ for the category of $\kappa $-lex $E$-preserving $\mathcal{C}\to \mathbf{Set}$ functors and natural transformations. $Mod(\mathcal{C})_{<\mu }$ denotes the full subcategory of models with pointwise cardinality $<\mu $. Similarly, for a $\kappa $-topos $\mathcal{E}$ we write $Mod_{\mathcal{E}}(\mathcal{C})=Mod_{\mathcal{E}}(\mathcal{C},E,\kappa )$ for the category of $\mathcal{C}\to \mathcal{E}$ $\kappa $-lex $E$-preserving functors.
\end{definition}

\begin{theorem}
\label{classiftopos}
    Let $(\mathcal{C},E)$ be a small $\kappa $-site, $\mathcal{E}$ be a $\kappa $-topos and $M:\mathcal{C}\to \mathcal{E}$ be a $\kappa $-lex $E$-preserving functor. Then
    \begin{enumerate}
        \item In the left Kan-extension
\[\begin{tikzcd}
	{\mathcal{C}} && {\mathcal{E}} \\
	\\
	{Sh(\mathcal{C},\langle E\rangle _{\kappa })}
	\arrow[""{name=0, anchor=center, inner sep=0}, "M", from=1-1, to=1-3]
	\arrow["ay"', from=1-1, to=3-1]
	\arrow["{Lan_{ay}M}"', from=3-1, to=1-3]
	\arrow["\eta "', shorten <=14pt, shorten >=14pt, Rightarrow, from=0, to=3-1]
\end{tikzcd}\]
$\eta $ is an isomorphism, and $Lan_{ay}M$ is $\kappa $-lex cocontinuous.
    \item $Lan_{ay}:Mod_{\mathcal{E}}(\mathcal{C},E,\kappa )\to Fun^*_{\kappa }(Sh(\mathcal{C},\langle E\rangle _{\kappa }), \mathcal{E})$ is an equivalence of categories, whose quasi-inverse is precomposing with $ay$.
    \end{enumerate}
\end{theorem}

\begin{proof}
    The above triangle can be written as
\[\begin{tikzcd}
	{\mathcal{C}} &&&& {\mathcal{E}} \\
	\\
	{\mathbf{Set}^{\mathcal{C}^{op}}} && {\mathbf{Set}^{\mathcal{C}^{op}}} \\
	\\
	{Sh(\mathcal{C},\langle E\rangle _{\kappa })}
	\arrow[""{name=0, anchor=center, inner sep=0}, "M", from=1-1, to=1-5]
	\arrow["y"', from=1-1, to=3-1]
	\arrow[""{name=1, anchor=center, inner sep=0}, Rightarrow, no head, from=3-1, to=3-3]
	\arrow["{Lan_yM}"{description}, from=3-3, to=1-5]
	\arrow["a"', from=3-1, to=5-1]
	\arrow["i"', from=5-1, to=3-3]
	\arrow["{\eta _1}"', shorten <=15pt, shorten >=15pt, Rightarrow, from=0, to=1]
	\arrow["{\eta _2}"', shorten <=14pt, shorten >=14pt, Rightarrow, from=1, to=5-1]
\end{tikzcd}\]

(To check the universal property: given $F:Sh(\mathcal{C})\to \mathcal{E}$ and a natural transformation $\gamma :M\Rightarrow Fay$, there's an induced natural transformation $\widetilde{\gamma }:Lan_yM\Rightarrow Fa$ fitting in the picture:

\[\begin{tikzcd}
	& {\mathcal{C}} && {\mathcal{E}} \\
	{\gamma =} \\
	& {\mathbf{Set}^{\mathcal{C}^{op}}} && {\mathbf{Set}^{\mathcal{C}^{op}}} && {Sh(\mathcal{C},\langle E\rangle _{\kappa })} \\
	{id=} \\
	& {Sh(\mathcal{C},\langle E\rangle _{\kappa })}
	\arrow[""{name=0, anchor=center, inner sep=0}, "M", from=1-2, to=1-4]
	\arrow["y"', from=1-2, to=3-2]
	\arrow[""{name=1, anchor=center, inner sep=0}, Rightarrow, no head, from=3-2, to=3-4]
	\arrow["a"', from=3-2, to=5-2]
	\arrow["i"{description}, from=5-2, to=3-4]
	\arrow["a"{description}, from=3-4, to=3-6]
	\arrow["F"', from=3-6, to=1-4]
	\arrow[""{name=2, anchor=center, inner sep=0}, "{Lan_yM}"{description}, from=3-4, to=1-4]
	\arrow[""{name=3, anchor=center, inner sep=0}, Rightarrow, no head, from=5-2, to=3-6]
	\arrow["{\eta _2}"', shorten <=14pt, shorten >=14pt, Rightarrow, from=1, to=5-2]
	\arrow["{\eta _1}"', shorten <=13pt, shorten >=13pt, Rightarrow, from=0, to=1]
	\arrow["{\widetilde{\gamma }}", shift right=3, shorten <=18pt, shorten >=12pt, Rightarrow, from=2, to=3-6]
	\arrow["{\varepsilon _2}","{\cong}"', shorten <=3pt, shorten >=3pt, Rightarrow, from=3-4, to=3]
\end{tikzcd}\]
hence the pasting of $\widetilde{\gamma }$ and $\varepsilon _2$ yields the unique splitting of $\gamma $ we were looking for.)

As both $i$ and $Lan_yM$ are $\kappa $-lex it follows that $Lan_{ay}M$ is $\kappa $-lex.

$Lan_yM$ has a right adjoint $\mathcal{E}(M(-),\bullet )$ which factors through $i$ iff $M$ is $E$-preserving (easy). In this case it is also a right adjoint for $Lan_yM\circ i$: $\mathcal{E}(Lan_yM(iF),X)\cong \mathbf{Set}^{\mathcal{C}^{op}}(iF,i\mathcal{E}(M(-),X))\cong Sh(\mathcal{C})(F,\mathcal{E}(M(-),X))$. We proved that $Lan_{ay}M$ is $\kappa $-lex cocontinuous hence we have the restricted adjunction

\[\begin{tikzcd}
	{Mod_{\mathcal{E}}(\mathcal{C},E,\kappa )} &&& {Fun^*_{\kappa }(Sh(\mathcal{C}),\mathcal{E})}
	\arrow[""{name=0, anchor=center, inner sep=0}, "{Lan_{ay}}", curve={height=-12pt}, from=1-1, to=1-4]
	\arrow[""{name=1, anchor=center, inner sep=0}, "{(ay)^{\circ }}", curve={height=-12pt}, from=1-4, to=1-1]
	\arrow["\dashv"{anchor=center, rotate=-90}, draw=none, from=0, to=1]
\end{tikzcd}\]

We want to prove $(ay)^{\circ }$ to be an equivalence. Then it follows that the adjunction is an adjoint equivalence (whose unit is $\eta $ in Claim 1.) and the proof will be complete.

In the composite
\[\begin{tikzcd}
	{Fun^*_{\kappa }(Sh(\mathcal{C}),\mathcal{E})} && {Fun^*_{\kappa }(\mathbf{Set}^{\mathcal{C}^{op}}, \mathcal{E})} && {\mathbf{Lex}_{\kappa }(\mathcal{C},\mathcal{E})}
	\arrow["{y^{\circ }}", from=1-3, to=1-5]
	\arrow["{a^{\circ }}", from=1-1, to=1-3]
\end{tikzcd}\]
$y^{\circ }$ is an equivalence. $a^{\circ }$ is fully faithful: given $\alpha :Fa\Rightarrow Ga$ its unique preimage is $F\xRightarrow{F\circ \varepsilon ^{-1}}Fai\xRightarrow{\alpha \circ i}Gai \xRightarrow{G\circ \varepsilon }G$. 

A $\kappa $-lex cocontinuous functor $N^*:\mathbf{Set}^{\mathcal{C}^{op}}\to \mathcal{E}$ is in the essential image of $a^{\circ }$ iff its right adjoint $N_*$ factors through $i$. $\Rightarrow $ is clear (adjoints compose). To see $\Leftarrow $ assume $N_*=i\circ \widetilde{N}_*$. As $i$ preserves and reflects all limits it follows that $\widetilde{N}_*$ is continuous, hence it has a left adjoint $\widetilde{N}^*$. By the uniqueness of adjoints $\widetilde{N}^*\circ a \cong N^*$. Hence $\widetilde{N}^*$ preserves all $<\kappa $-limits which are in the image of $a$, that is, all $<\kappa $-limits as $ai\cong 1_{Sh(\mathcal{C})}$.

But for a $\kappa $-lex functor $M:\mathcal{C}\to \mathcal{E}$, the left Kan-extension $Lan_yM=(y^{\circ })^{-1}(M)$ satisfies this property iff $M$ was $E$-preserving, as we claimed before.

\end{proof}

\section{Eventually: enough points $\Rightarrow $ presheaf type}

This section proves the first half of the main theorem of these notes: that the classifying $\lambda $-topos of a $\kappa $-site is a presheaf topos. Here we will prove it under the assumption that the classifying $\lambda $-topos has enough $\lambda $-points. Then in section 5.~we shall see that this is automatic when $\lambda ^{<\lambda }=\lambda $.

\begin{remark}
    As previously mentioned, $\kappa $-sites are the same as theories in $L^g_{\infty \kappa }$. Since $L^g_{\infty \kappa }\subseteq L^g_{\infty \lambda }$ (for $\lambda \geq \kappa $), there is an easy way to define the free $\lambda $-site associated to some $\kappa $-site $(\mathcal{C},E)$: write down the theory of $(\mathcal{C},E)$ over its canonical signature, treat it as a theory in $L^g_{\infty \lambda }$, then take the syntactic $\lambda $-site (cf.~\cite[Chapter 3]{accessible}).

    However, we prefer to remain within the framework of category theory (while retaining model-theoretic intuition) for a cleaner and more rigorous discussion.
    
    In Theorem \ref{ptsthenpresheaf}, we will provide an explicit and simple description of the associated $\lambda$-site, at least when $\lambda$ is "sharply larger" than $\kappa$. For now, we just call it $(\widetilde{\mathcal{C}},\widetilde{E})$ and include its "freeness" properties as assumptions.
\end{remark}
 
\begin{lemma}\label{keylemma}
We have the following assumptions:
\begin{enumerate}
    \item $\lambda =cf(\lambda )\geq \kappa =cf(\kappa )\geq \aleph _0$.
    \item $(\mathcal{C},E)$ is a small $\kappa $-site, $(\widetilde{\mathcal{C}},\widetilde{E})$ is a small $\lambda $-site, $\varphi :(\mathcal{C},E)\to (\widetilde{\mathcal{C}},\langle \widetilde{E}\rangle _{\lambda })$ is a morphism of $\kappa $-sites. 
    \item $|\mathcal{C}|< \lambda $.
    \item For any $\lambda $-topos $\mathcal{E}$, the map $\varphi ^{\circ }:Mod_{\mathcal{E}}(\widetilde{\mathcal{C}},\widetilde{E},\lambda )\to Mod_{\mathcal{E}}(\mathcal{C},E,\kappa )$ is an equivalence of categories.
    \item There are $(\widetilde{N_i} :\widetilde{\mathcal{C}}\to \mathbf{Set})_i$ $\lambda $-lex $\widetilde{E}$-preserving functors with $\widetilde{N_i}\varphi $ having pointwise size $<\lambda $, such that $\langle Lan_{ay}\widetilde{N_i}\rangle _i:Sh(\widetilde{\mathcal{C}},\langle \widetilde{E} \rangle _{\lambda })\to \mathbf{Set}^I$ is conservative.
\end{enumerate}
Then there is an equivalence $\mathbf{Set}^{Mod(\mathcal{C})_{<\lambda }}\to Sh(\widetilde{\mathcal{C}}, \langle \widetilde{E}\rangle _{\lambda })$ making
\[
\adjustbox{scale=1}{
\begin{tikzcd}
	{\mathcal{C}} &&& {\mathbf{Set}^{Mod(\mathcal{C})_{<\lambda }}} \\
	{\widetilde{\mathcal{C}}} \\
	\\
	{Sh(\widetilde{\mathcal{C}},\langle \widetilde{E}\rangle _{\lambda})}
	\arrow["\varphi"', from=1-1, to=2-1]
	\arrow["ay"', from=2-1, to=4-1]
	\arrow[""{name=0, anchor=center, inner sep=0}, "ev", from=1-1, to=1-4]
	\arrow[""{name=1, anchor=center, inner sep=0}, "\simeq", from=1-4, to=4-1]
	\arrow["\cong"', draw=none, from=0, to=1]
\end{tikzcd}
}
\]
commutative.
\end{lemma}

\begin{proof}
    We define a functor $\Delta :Mod(\mathcal{C})_{<\lambda }^{op}\to \widetilde{\mathcal{C}}$ by 
\[\begin{tikzcd}
	{M\xrightarrow{\alpha}N} & \mapsto & {lim_{(x,p)\in \int M} \ \varphi x} && {lim_{(x,q)\in \int N} \ \varphi x}
	\arrow["{\langle \pi _{(x,\alpha _x(p)) }\rangle }"', from=1-5, to=1-3]
\end{tikzcd}\]
which makes sense as by the regularity of $\lambda $: $|\int M |<\lambda $. We have a natural transformation:

\[
\adjustbox{scale=1}{
\begin{tikzcd}
	{\mathcal{C}} &&& {\mathbf{Set}^{Mod(\mathcal{C})_{<\lambda }}} \\
	{\widetilde{\mathcal{C}}} \\
	\\
	{Sh(\widetilde{\mathcal{C}},\langle \widetilde{E}\rangle _{\lambda})} &&& \textcolor{rgb,255:red,128;green,128;blue,128}{Mod(\mathcal{C})_{<\lambda }^{op}}
	\arrow["\varphi"', from=1-1, to=2-1]
	\arrow["ay"', from=2-1, to=4-1]
	\arrow["ev", from=1-1, to=1-4]
	\arrow[""{name=0, anchor=center, inner sep=0}, "{Lan_y(ay\Delta)}"{description}, from=1-4, to=4-1]
	\arrow["y"', color={rgb,255:red,128;green,128;blue,128}, hook, from=4-4, to=1-4]
	\arrow[""{name=1, anchor=center, inner sep=0}, "ay\Delta", color={rgb,255:red,128;green,128;blue,128}, from=4-4, to=4-1]
	\arrow["\eta"{description}, shorten <=27pt, shorten >=18pt, Rightarrow, from=1-4, to=2-1]
	\arrow["\cong"', color={rgb,255:red,128;green,128;blue,128}, shorten <=10pt, shorten >=10pt, Rightarrow, from=1, to=0]
\end{tikzcd}
}
\]
whose $v$-component is defined by

\[
\adjustbox{scale=1}{
\begin{tikzcd}
	& {a\widetilde{\mathcal{C}}(-,\varphi v)} \\
	\\
	& {Lan_y(ay\Delta)(ev_v)} \\
	\\
	{\underbrace{lim _{(x,p)\in \int M} \ a\widetilde{\mathcal{C}}(-,\varphi x)}_{(M,\ p_0\in Mv)}} && {\underbrace{lim _{(x,q)\in \int N} \ a\widetilde{\mathcal{C}}(-,\varphi x)}_{(N,\ q_0\in Nv)}}
	\arrow["{\alpha }"{description}, color={rgb,255:red,128;green,128;blue,128}, curve={height=12pt}, shorten <=8pt, shorten >=8pt, from=5-1, to=5-3]
	\arrow["{\langle \pi _{(x,\alpha _x(p))}\rangle}"', shorten <=8pt, shorten >=8pt, from=5-3, to=5-1]
	\arrow["{\pi _{(v,p_0)}}", curve={height=-12pt}, from=5-1, to=1-2]
	\arrow["{\pi _{(v,q_0)}}"', curve={height=12pt}, from=5-3, to=1-2]
	\arrow[from=5-1, to=3-2]
	\arrow[from=5-3, to=3-2]
	\arrow["{\eta _v}"', dashed, from=3-2, to=1-2]
\end{tikzcd}
}
\]
(i.e.~to compute $Lan_y(ay\Delta)(ev_v)$ we have to write $ev_v$ as the colimit of representables along its category of elements, then apply $ay\Delta $ to this diagram, finally compute its colimit. We have a cocone over this diagram with top $ay\varphi (v)$, and $\eta _v$ is the induced map).

To check the commutativity of the naturality squares (say at $f:v\to w$), one has to precompose with a leg of the colimit, then it has the form:

\[
\adjustbox{scale=1}{
\begin{tikzcd}
	{a\widetilde{\mathcal{C}}(-,\varphi v)} &&& {a\widetilde{\mathcal{C}}(-,\varphi w)} \\
	\\
	{Lan_y(ay\Delta)(ev_v)} &&& {Lan_y(ay\Delta)(ev_w)} \\
	\\
	{\underbrace{lim _{(x,p)\in \int M} \ a\widetilde{\mathcal{C}}(-,\varphi x)}_{(M,\ p_0\in Mv)}} &&& {\underbrace{lim _{(x,p)\in \int M} \ a\widetilde{\mathcal{C}}(-,\varphi x)}_{(M,\ Mf(p_0)\in Mw)}}
	\arrow[shorten <=23pt, shorten >=23pt, Rightarrow, no head, from=5-1, to=5-4]
	\arrow[from=5-1, to=3-1]
	\arrow[dashed, from=3-1, to=3-4]
	\arrow[from=5-4, to=3-4]
	\arrow["{\eta _v}"', from=3-1, to=1-1]
	\arrow["{\eta _w}"', from=3-4, to=1-4]
	\arrow["{a(\varphi f_{\circ })}"{description}, from=1-1, to=1-4]
	\arrow["{\pi _{(v,p_0)}}"{pos=0.7}, curve={height=-60pt}, from=5-1, to=1-1]
	\arrow["{\pi _{(w,Mf(p_0))}}"'{pos=0.7}, curve={height=60pt}, from=5-4, to=1-4]
\end{tikzcd}
}
\]
which can be drawn as

\[
\adjustbox{scale=1}{
\begin{tikzcd}
	& {lim _{(x,p)\in \int M}\ a\widetilde{\mathcal{C}}(-,\varphi x)} \\
	\\
	{a\widetilde{\mathcal{C}}(-,\varphi v)^{p_0}} && {a\widetilde{\mathcal{C}}(-,\varphi w)^{Mf(p_0)}}
	\arrow["{\pi _{(v,p_0)}}"', from=1-2, to=3-1]
	\arrow["{a(\varphi f_{\circ })}", from=3-1, to=3-3]
	\arrow["{\pi _{(w,Mf(p_0))}}", from=1-2, to=3-3]
\end{tikzcd}
}
\]

Our goal is to prove that $\eta $ is an isomorphism and $Lan_y(ay\Delta )$ is an equivalence. We start with the following observation (saying that if $M$ is a sufficiently small model, so that its diagram $\Delta M$ is a single formula in $\widetilde{\mathcal{C}}$, then its evaluation at a model $\widetilde{N}(\Delta M)$ consists of those tuples which enumerate an $M\to \widetilde{N}\varphi $ homomorphism):

There is a natural isomorphism
\[
\adjustbox{scale=1}{
\begin{tikzcd}
	{Mod(\mathcal{C})_{<\lambda }^{op}} && {\widetilde{\mathcal{C}}} \\
	\\
	{\mathbf{Set}^{Mod(\mathcal{C})_{<\lambda }}} && {\mathbf{Set}^{Mod(\widetilde{\mathcal{C}})_{-\circ \varphi <\lambda }}}
	\arrow["\Delta", from=1-1, to=1-3]
	\arrow["y"', hook', from=1-1, to=3-1]
	\arrow["{(\varphi ^{\circ })^{\circ }}"', from=3-1, to=3-3]
	\arrow["ev", from=1-3, to=3-3]
	\arrow["\delta"{description}, shorten <=22pt, shorten >=22pt, Rightarrow, from=1-3, to=3-1]
\end{tikzcd}
}
\]
given by

\[
\adjustbox{width=\textwidth}{
\begin{tikzcd}
	{ev_{\Delta M}} & {lim_{(x,p)\in \int M}\ ev_{\varphi x}} & {lim_{\int M}\ Nat(\mathcal{C}(-,x),- \circ \varphi )} & {} \\
	\\
	{ev_{\Delta N}} & {lim_{(x,q)\in \int N}\ ev_{\varphi x}} & {lim_{\int N}\ Nat(\mathcal{C}(-,x),- \circ \varphi )} & {} \\
	& {} & {Nat(colim_{\int M } \mathcal{C}(-,x),-\circ \varphi )} & {Nat(M,-\circ \varphi )} \\
	\\
	& {} & {Nat(colim_{\int N } \mathcal{C}(-,x),-\circ \varphi )} & {Nat(N,-\circ \varphi )}
	\arrow["{ev_{\Delta \alpha }}", from=3-1, to=1-1]
	\arrow["\cong"{description}, draw=none, from=1-1, to=1-2]
	\arrow["{\langle \pi _{(x,\alpha _xp)}\rangle}", from=3-2, to=1-2]
	\arrow["\cong"{description}, draw=none, from=3-1, to=3-2]
	\arrow["\cong"{description}, draw=none, from=1-2, to=1-3]
	\arrow["\cong"{description}, draw=none, from=3-2, to=3-3]
	\arrow["{\langle \pi _{(x,\alpha _xp)}\rangle}", from=3-3, to=1-3]
	\arrow["{(i_{(x,\alpha _xp)})^{\circ }}", from=6-3, to=4-3]
	\arrow["\cong"{description}, draw=none, from=4-3, to=4-4]
	\arrow["{\alpha ^{\circ }}"', from=6-4, to=4-4]
	\arrow["\cong"{description}, draw=none, from=6-3, to=6-4]
	\arrow["\cong"{description}, draw=none, from=1-3, to=1-4]
	\arrow["\cong"{description}, draw=none, from=3-3, to=3-4]
	\arrow["\cong"{description}, draw=none, from=4-2, to=4-3]
	\arrow["\cong"{description}, draw=none, from=6-2, to=6-3]
\end{tikzcd}
}
\]

Now we will prove that $\eta $ is an isomorphism. For this it suffices to prove that 

\[\begin{tikzcd}
	& {\varphi v} \\
	\\
	{\underbrace{lim _{(x,p)\in \int M} \ \varphi x}_{(M,\ p_0\in Mv)}} && {\underbrace{lim _{(x,q)\in \int N} \ \varphi x}_{(N,\ q_0\in Nv)}}
	\arrow["{\alpha }"{description}, color={rgb,255:red,128;green,128;blue,128}, curve={height=12pt}, shorten <=4pt, shorten >=4pt, from=3-1, to=3-3]
	\arrow["{\langle \pi _{(x,\alpha _x(p))}\rangle}"', shorten <=4pt, shorten >=4pt, from=3-3, to=3-1]
	\arrow["{\pi _{(v,p_0)}}", curve={height=-12pt}, from=3-1, to=1-2]
	\arrow["{\pi _{(v,q_0)}}"', curve={height=12pt}, from=3-3, to=1-2]
\end{tikzcd}\]
is mapped to a colimit by $ay$. By Theorem \ref{classiftopos} and by Assumption 5.~it is enough that it is mapped to a colimit by any $\widetilde{N}:\widetilde{\mathcal{C}}\to \mathbf{Set}$ $\lambda $-lex $\widetilde{E}$-preserving functor with $|\widetilde{N}\varphi |<\lambda $ (as in this case the $colim\to ay\varphi (v)$ map is taken to an iso by each (cocontinuous) $Lan_{ay}\widetilde{N}$, and those are jointly conservative).

But using the isomorphism $\delta $ constructed above we have 
\[\begin{tikzcd}
	{\widetilde{N}\varphi v} && {Nat(\mathcal{C}(-,v),\widetilde{N}\varphi )} \\
	& {=} \\
	{\underbrace{\widetilde{N}(\Delta M)=lim _{(x,p)\in \int M} \ \widetilde{N}\varphi x}_{(M,\ p_0\in Mv)}} && {\underbrace{Nat(M,\widetilde{N}\varphi )}_{(M,p_0\in Mv)}}
	\arrow["{\pi _{(v,p_0)}}", from=3-1, to=1-1]
	\arrow["\cong ","{(\delta _M)_{\widetilde{N}}}"', from=3-1, to=3-3]
	\arrow["{i_{(v,p_0)}^{\circ } }"', from=3-3, to=1-3]
	\arrow["\cong", from=1-1, to=1-3]
\end{tikzcd}\]
and the composite $Nat(M,\widetilde{N}\varphi )\to \widetilde{N}\varphi v$ takes $\beta $ to $\beta _v(p_0)$. So this is the canonical colimit $ev_{v}\cong colim _{(M,p_0)\in \int ev_v}Mod(\mathcal{C})_{<\lambda }(M,-)$, evaluated at $\widetilde{N}\varphi $.

It remains to check that $Lan_y(ay\Delta )$ is an equivalence. Its proposed quasi-inverse is $((\varphi ^{\circ })^{\circ })^{-1}\circ Lan_{ay} ev$, fitting in the diagram
\[\begin{tikzcd}
	{\mathbf{Set}^{Mod(\mathcal{C})_{<\lambda }}} && {\mathbf{Set}^{Mod(\widetilde{\mathcal{C}})_{-\circ \varphi <\lambda }}} \\
	\\
	{Mod(\mathcal{C})_{<\lambda }^{op}} && {\widetilde{\mathcal{C}}} && {Sh(\widetilde{C},\langle \widetilde{E}\rangle _{\lambda})}
	\arrow["\Delta", from=3-1, to=3-3]
	\arrow["y", hook, from=3-1, to=1-1]
	\arrow["{(\varphi ^{\circ })^{\circ }}", from=1-1, to=1-3]
	\arrow[""{name=0, anchor=center, inner sep=0}, "ev", from=3-3, to=1-3]
	\arrow["\delta"', shorten <=22pt, shorten >=22pt, Rightarrow, from=3-3, to=1-1]
	\arrow["ay", from=3-3, to=3-5]
	\arrow["{Lan_{ay}ev}"', curve={height=12pt}, from=3-5, to=1-3]
	\arrow["\gamma ^{-1}"', shorten <=28pt, shorten >=14pt, Rightarrow, from=3-5, to=0]
\end{tikzcd}\]

Since both $Lan_y(ay\Delta )$ and $((\varphi ^{\circ })^{\circ })^{-1}\circ Lan_{ay} ev$ are cocontinuous, this iso extends to an isomorphism between identity on $\mathbf{Set}^{Mod(\mathcal{C})_{<\lambda }}$ and $((\varphi ^{\circ })^{\circ })^{-1}\circ Lan_{ay} ev \circ Lan_y(ay\Delta )$.

To check the other composite, observe the diagram

\[
\adjustbox{width=\textwidth}{
\begin{tikzcd}
	&&&& {\mathcal{C}} \\
	&& {\widetilde{\mathcal{C}}} &&&& {\widetilde{\mathcal{C}}} \\
	&&& \cong \\
	{Sh(\widetilde{\mathcal{C}},\langle \widetilde{E}\rangle _{\lambda})} && {\mathbf{Set}^{Mod(\widetilde{\mathcal{C}})_{-\circ \varphi <\lambda }}} && {\mathbf{Set}^{Mod(\mathcal{C})_{<\lambda }}} && {Sh(\widetilde{\mathcal{C}},\langle \widetilde{E}\rangle _{\lambda})}
	\arrow["\varphi", from=1-5, to=2-7]
	\arrow["ay", from=2-7, to=4-7]
	\arrow["ev"{description}, from=1-5, to=4-5]
	\arrow["{Lan_y(ay\Delta)}"', from=4-5, to=4-7]
	\arrow["\eta", shorten <=21pt, shorten >=14pt, Rightarrow, from=4-5, to=2-7]
	\arrow["{((\varphi ^{\circ })^{\circ })^{-1}}"', from=4-3, to=4-5]
	\arrow["{\varphi }"', from=1-5, to=2-3]
	\arrow[""{name=0, anchor=center, inner sep=0}, "ev"{description}, from=2-3, to=4-3]
	\arrow["ay"', from=2-3, to=4-1]
	\arrow[""{name=1, anchor=center, inner sep=0}, "{Lan_{ay}ev}"', from=4-1, to=4-3]
	\arrow["{\gamma ^{-1}}", shorten <=10pt, shorten >=10pt, Rightarrow, from=1, to=0]
\end{tikzcd}
}
\]
Assume that $Lan_y(ay\Delta )$ preserves $<\lambda $-limits. Then both $ay$ and $Lan_y(ay\Delta )\circ ((\varphi ^{\circ })^{\circ })^{-1}\circ Lan_{ay} ev \circ ay$ are $\widetilde{\mathcal{C}}\to Sh(\widetilde{\mathcal{C}},\langle \widetilde{E} \rangle _{\lambda }) $ $\lambda $-lex $\widetilde{E}$-preserving functors, and there's an iso between their $\varphi $-restrictions. By Assumption 4.~$\varphi ^{\circ }$ is fully faithful, so there's an iso between the $\widetilde{\mathcal{C}}\to  Sh(\widetilde{\mathcal{C}},\langle \widetilde{E} \rangle _{\lambda })$ functors. So identity on $Sh(\widetilde{\mathcal{C}},\langle \widetilde{E} \rangle _{\lambda })$ and $Lan_y(ay\Delta )\circ ((\varphi ^{\circ })^{\circ })^{-1}\circ Lan_{ay} ev $ are two $\lambda $-topos maps whose $ay$-restrictions are isomorphic. Then Theorem \ref{classiftopos} Claim 2.~completes the proof.

It remains to check that $Lan_y(ay\Delta)$ preserves $<\lambda $-limits. For any $\widetilde{N}:\widetilde{\mathcal{C}}\to \mathbf{Set}$ $\lambda $-lex $\widetilde{E}$-preserving with $|\widetilde{N}\varphi |<\lambda $ we have
\[\begin{tikzcd}
	{\mathbf{Set}^{Mod(\mathcal{C})_{<\lambda}}} \\
	\\
	{Mod(\mathcal{C})_{<\lambda}^{op}} && {Sh(\widetilde{\mathcal{C}},\langle \widetilde{E}\rangle _{\lambda})} && {\mathbf{Set}}
	\arrow["y", hook, from=3-1, to=1-1]
	\arrow[""{name=0, anchor=center, inner sep=0}, "ay\Delta"{description}, from=3-1, to=3-3]
	\arrow["{Lan_y(ay\Delta)}", from=1-1, to=3-3]
	\arrow["{\widetilde{N}^*}"{description}, from=3-3, to=3-5]
	\arrow[""{name=1, anchor=center, inner sep=0}, "{Mod(\mathcal{C})_{<\lambda}(-,\widetilde{N}\varphi )}"', curve={height=24pt}, from=3-1, to=3-5]
	\arrow["\cong", shorten <=15pt, shorten >=15pt, Rightarrow, from=0, to=1-1]
	\arrow["\cong"{description}, draw=none, from=1, to=3-3]
\end{tikzcd}\]
coming from
\[\begin{tikzcd}
	{Mod(\mathcal{C})_{<\lambda }^{op}} && {\widetilde{\mathcal{C}}} && {Sh(\widetilde{\mathcal{C}})} \\
	{\mathbf{Set}^{Mod(\mathcal{C})_{<\lambda }}} && {\mathbf{Set}^{Mod(\widetilde{\mathcal{C}})_{-\circ \varphi <\lambda }}} && {\mathbf{Set}}
	\arrow[""{name=0, anchor=center, inner sep=0}, "\Delta", from=1-1, to=1-3]
	\arrow["y"', from=1-1, to=2-1]
	\arrow[""{name=1, anchor=center, inner sep=0}, "ay", from=1-3, to=1-5]
	\arrow["ev", from=1-3, to=2-3]
	\arrow["{\widetilde{N}^*}", from=1-5, to=2-5]
	\arrow[""{name=2, anchor=center, inner sep=0}, "{(\varphi ^{\circ })^{\circ}}"', from=2-1, to=2-3]
	\arrow[""{name=3, anchor=center, inner sep=0}, "{ev_{\widetilde{N}}}"', from=2-3, to=2-5]
	\arrow["{\cong _{\delta }}"{description}, draw=none, from=0, to=2]
	\arrow["\cong"{description}, draw=none, from=1, to=3]
\end{tikzcd}\]

As left adjoints preserve left Kan-extensions, $\widetilde{N}^*\circ Lan_y(ay\Delta )$ is the left Kan-extension of a representable, hence representable, in particular $\lambda $-lex. So $\mathbf{Set}^{Mod(\mathcal{C})_{<\lambda}}\xrightarrow{Lan_y(ay\Delta)} Sh(\widetilde{\mathcal{C}},\langle \widetilde{E}\rangle _{\lambda}) \xrightarrow{ \widetilde{N_i}^*}\mathbf{Set}^I$ is $\lambda $-lex, the second map is $\lambda $-lex, conservative, therefore the first map is $\lambda $-lex (the second map inverts the connecting homomorphism going from the image of the limit to the limit of the images).
\end{proof}

Now we would like to identify $\widetilde{\mathcal{C}}$ from the previous lemma. Assume that we can find $\varphi :\mathcal{C}\to \widetilde{\mathcal{C}}$ which is the free completion of $\mathcal{C}$ under $<\lambda $-limits, in the sense that $\varphi ^{\circ }:\mathbf{Lex}_{\lambda }(\widetilde{\mathcal{C}},\mathcal{E})\to \mathbf{Lex}_{\kappa }(\mathcal{C},\mathcal{E})$ is an equivalence for any $\lambda $-lex (or at least for complete) $\mathcal{E}$. Then it restricts to an equivalence between the full subcategories of $\varphi[E]$-preserving and $E$-preserving functors. 

We have a good candidate: assume it exists, then apply the previous lemma with trivial $E$ (only the identities are contained). We get an equivalence: $\mathbf{Set}^{\widetilde{\mathcal{C}}^{op}}\simeq \mathbf{Set}^{\mathbf{Lex}_{\kappa }(\mathcal{C},\mathbf{Set})_{<\lambda }}$. So our guess is that $y:\mathcal{C}\to (\mathbf{Lex}_{\kappa }(\mathcal{C},\mathbf{Set})_{<\lambda })^{op}$ will do the job. 

The following is \cite[Proposition 1.45 (ii)]{rosicky}:

\begin{proposition}
    $\mathcal{C}$ is $\kappa $-lex, small. Take $y:\mathcal{C}\hookrightarrow \mathbf{Lex}_{\kappa }(\mathcal{C},\mathbf{Set})^{op}$. Then $y$ is $\kappa $-lex and $y^{\circ }:\mathbf{LEX}_{\infty }(\mathbf{Lex}_{\kappa }(\mathcal{C},\mathbf{Set})^{op}, \mathcal{E})\to \mathbf{Lex}_{\kappa }(\mathcal{C},\mathcal{E})$ is an equivalence for any complete category $\mathcal{E}$.
\end{proposition}

\begin{proof}
    $y$ preserves $<\kappa $-limits: we need that $\mathcal{C}(lim\ x_i,-)\xleftarrow{\pi _j^{\circ }} \mathcal{C}(x_j,-)$ is a colimit diagram in $\mathbf{Lex}_{\kappa }(\mathcal{C},\mathbf{Set})$. A cocone with top $M$ is a compatible family of elements $(a_i\in Mx_i)_i$ in the $M$-image of the base diagram. This corresponds to a unique element $(a_i)_i$ in $M(lim\ x_i)$ as $M$ preserves $<\kappa $-limits, which yields the unique induced arrow $\mathcal{C}(lim\ x_i,-)\xRightarrow{1\mapsto (a_i)_i}M$.

    We claim that for an $M:\mathcal{C}\to \mathcal{E}$ $\kappa $-lex functor (for complete $\mathcal{E}$), the right Kan-extension
\[\begin{tikzcd}
	{\mathcal{C}} && {\mathcal{E}} \\
	\\
	{\mathbf{Lex}_{\kappa }(\mathcal{C},\mathbf{Set})^{op}}
	\arrow["y"', hook', from=1-1, to=3-1]
	\arrow[""{name=0, anchor=center, inner sep=0}, "M", from=1-1, to=1-3]
	\arrow["{Ran_yM}"', from=3-1, to=1-3]
	\arrow["\varepsilon", "\cong "', shorten <=14pt, shorten >=14pt, Rightarrow, from=3-1, to=0]
\end{tikzcd}\]
    is continuous. This follows as $\mathcal{E}\ni x \mapsto \mathcal{E}(x,M(-))\in \mathbf{Lex}_{\kappa }(\mathcal{C},\mathbf{Set})^{op}$ is a left adjoint to it.

    So $y^{\circ }$ is essentially surjective. It is also fully faithful: given continuous functors $\widetilde{M},\widetilde{N}: \mathbf{Lex}_{\kappa }(\mathcal{C},\mathbf{Set})^{op}\to \mathcal{E}$, they are the right Kan-extensions of $\widetilde{M}y$, resp.~$\widetilde{N}y$ (with identity as $\varepsilon$), hence any natural transformation $\alpha :\widetilde{M}y\Rightarrow \widetilde{N}y$ induces a unique $\widetilde{\alpha }:\widetilde{M}\to \widetilde{N}$ for which $\widetilde{\alpha }\circ y=\alpha $.
\end{proof}

\begin{definition}
    $\kappa =cf(\kappa )$ is \emph{sharply smaller} than $\lambda =cf(\lambda )$ (written as $\kappa \vartriangleleft \lambda $) if $\kappa <\lambda $ and for any set $X$ with $|X|<\lambda $, the poset $\mathcal{P}_{\kappa }(X)$ of $<\kappa $ subsets contains a cofinal set of size $<\lambda $. 
\end{definition}

\begin{proposition}\label{freelambdalim}
    $\mathcal{C}$ is $\kappa $-lex, small, $\aleph _0\leq \kappa = cf(\kappa ) \vartriangleleft \lambda =cf(\lambda )$. Write $\mathcal{C}\xhookrightarrow{\varphi }\widetilde{\mathcal{C}}\xhookrightarrow{j} \mathbf{Lex}_{\kappa }(\mathcal{C},\mathbf{Set})^{op}$ for the factorization of $y$ through the full subcategory spanned by $\kappa $-cofiltered $<\lambda $ limits of representables. Then $\widetilde{\mathcal{C}}$ is $\lambda $-lex, $\varphi $ is $\kappa $-lex and $\varphi ^{\circ }:\mathbf{Lex}_{\lambda }(\widetilde{\mathcal{C}},\mathcal{E})\to \mathbf{Lex}_{\kappa }(\mathcal{C},\mathcal{E})$ is an equivalence for any complete category $\mathcal{E}$.
\end{proposition}

\begin{proof}
We claim that $\widetilde{\mathcal{C}}$ is the full subcategory spanned by $<\lambda $ limits of representables. Indeed, take a $<\lambda $ diagram $\mathcal{I}\to \mathbf{Lex}_{\kappa }(\mathcal{C},\mathbf{Set})$. Its colimit is the $\kappa $-directed colimit indexed by $\mathcal{P}_{\kappa }(Arr(\mathcal{I}))$, of the colimits of $<\kappa $ subdiagrams and the induced maps between them. These $<\kappa $-colimits are representable (see the first paragraph of the previous proof) and by assumption the diagram has a cofinal subdiagram of size $<\lambda $. It follows that $\widetilde{\mathcal{C}}$ is $\lambda $-lex and $\varphi $ is $\kappa $-lex.

$\varphi ^{\circ }$ is essentially surjective: given $M:\mathcal{C}\to \mathcal{E}$ $\kappa $-lex $Ran_yM \circ j$ gives a preimage. It is faithful: any natural transformation between $\widetilde{M},\widetilde{N}:\widetilde{\mathcal{C}}\to \mathcal{E}$ $\lambda $-lex maps is uniquely determined by its $\varphi $-restriction; the components at the $<\lambda $-limits are the induced morphisms.

$\varphi ^{\circ }$ is full: to simplify notation note that since $y$ is injective on objects, $\varepsilon $ in the right Kan-extension can be assumed to be an identity. Take $\widetilde{M},\widetilde{N}:\widetilde{\mathcal{C}}\to \mathcal{E}$ $\lambda $-lex and a natural transformation $\alpha :\widetilde{M}\varphi \Rightarrow \widetilde{N}\varphi $. Then $Ran_y(\widetilde{M}\varphi )\circ j$ is a map whose $\varphi $-restriction equals $\widetilde{M}\varphi $ and similarly for $\widetilde{N}$, moreover $\alpha $ has a preimage between these functors. So it suffices to prove that if $\widetilde{M_1}$ and $\widetilde{M_2}$ are $\widetilde{\mathcal{C}}\to \mathcal{E}$ $\lambda $-lex functors such that $\widetilde{M_1}\varphi =\widetilde{M_2}\varphi $ then there's a natural isomorphism $\gamma :\widetilde{M_1}\Rightarrow \widetilde{M_2}$ such that $\restr{\gamma }{\mathcal{C}}$ is identity. But given an arrow $f:x\to y$ in $\widetilde{\mathcal{C}}^{op}$ its domain and codomain are both $\kappa $-filtered colimits of representables, which is also a colimit in $\mathbf{Set}^{\mathcal{C}}$ where representables are tiny, so the restrictions of our arrow factor through some leg of the colimit with top $y$. It follows that once $\widetilde{M_1}$ decides to which limit object $x$ and $y$ would be sent, the image of $f$ is uniquely determined. The comparison maps between the limits chosen by $M_1$ and those chosen by $M_2$ give the required isomorphism. 
\end{proof}

As a last ingredient, we will prove the following version of the downward Löwenheim-Skolem theorem:

\begin{theorem}\label{LS}
    Assume $\aleph _0\leq \kappa =cf(\kappa )\vartriangleleft \lambda =cf(\lambda )$. Let $(\mathcal{C},E)$ be a $\kappa $-site with $|\mathcal{C}|,|E|<\lambda $. Then $Mod(\mathcal{C},E,\kappa )$ is a $\lambda $-accessible category with $\kappa $-filtered colimits (and its $\lambda $-presentable objects are the pointwise $<\lambda $ functors).
\end{theorem}

\begin{proof}
    It is enough to see that in $Mod(\mathcal{C},E,\kappa )$ every object is a $\lambda $-filtered colimit of pointwise $<\lambda $ models. So start with a model $M:\mathcal{C}\to \mathbf{Set}$. In $\mathbf{Set}^{\mathcal{C}}$ it is the $\lambda $-filtered union of its pointwise $<\lambda $ subfunctors, as the subfunctor generated by a collection of subsets $A(x)\subseteq M(x)$ is the closure under $M(f)$-images for $f\in \mathcal{C}$ but $|\mathcal{C}| <\lambda $. A filtered union is a filtered colimit, because this holds in $\mathbf{Set}$ and in presheaf categories colimits are pointwise. So it suffices to prove that every pointwise $<\lambda $ subfunctor of $M$ is contained in a pointwise $<\lambda $ subfunctor which is $\kappa $-lex $E$-preserving.

    We will prove the following: i) every $<\lambda $ subfunctor is contained in a $<\lambda $ subfunctor which is lex and $E$-preserving, ii) it is contained in a $<\lambda $ subfunctor which preserves $<\kappa $ products. This is sufficient as we can build a $\kappa $-chain out of $A$, in odd steps applying i), in even steps applying ii), in limit steps taking the colimit, then the union of this chain is $<\lambda $, $E$-preserving, and preserves both finite limits and $<\kappa $-products, hence $\kappa $-lex.

    i) This is easy: one has to build an $\omega $-chain out of $A$, in each step applying the following three closure operators:  
    \begin{itemize}
        \item[-] close under $M(f)$-images for $f\in \mathcal{C}$,
        \item[-] for each family $(u_i\to x)_i$ in $E$, and each $a\in A_n(x)$ choose one arbitrary preimage in one $M(u_j)$ and add it to $A_n(u_j)$,
        \item[-] for each finite diagram $\mathcal{I}\to \mathcal{C}$ and each compatible family formed by elements in $A_n$ ($(<\lambda )^{<\omega }=(<\lambda ))$ add the corresponding element to $A_n(lim_{\mathcal{I}})$.
    \end{itemize}
    Each of these add $<\lambda $ elements so $A_{\omega }$ satisfies our requirements. 

    ii) By \cite[Remark 1.21]{rosicky} we can write $M$ as a $\kappa $-directed (as opposed to just $\kappa $-filtered) colimit of representables (indexed by $P$). By \cite[Theorem 2.11]{rosicky} $\kappa \vartriangleleft \lambda $ implies that the $\kappa $-directed $<\lambda $ subsets of $P$ form a $\lambda $-directed poset $Q$. Hence we can write $M$ as a $\lambda $-directed colimit of pointwise $<\lambda $ and $\kappa $-lex functors (indexed by $Q$), as representables are pointwise $ <\lambda $ and therefore their $<\lambda $ colimits are $<\lambda $. We can take the pointwise image-factorization of each $M_q\Rightarrow M$ hence $M$ is the $\lambda $-directed union of these images which are easily proved to preserve $<\kappa $ products.
\end{proof}

\begin{theorem}\label{ptsthenpresheaf}
We have the following assumptions:
    \begin{itemize}
        \item $\aleph _0\leq \kappa =cf(\kappa )\vartriangleleft \lambda =cf(\lambda )$
        \item $(\mathcal{C},E)$ is a $\kappa $-site, $|\mathcal{C}|,|E| <\lambda  $.
    \end{itemize}
Write $\mathcal{C}\xhookrightarrow{\varphi } \widetilde{\mathcal{C}}$ for  $\mathcal{C}\xhookrightarrow{y}(\mathbf{Lex}_{\kappa }(\mathcal{C},\mathbf{Set})_{<\lambda })^{op}$. Then:
    \begin{itemize}
        \item  The map $-\circ \varphi :Mod_{\mathcal{E}}(\widetilde{\mathcal{C}},\varphi [E],\lambda )\to Mod_{\mathcal{E}}(\mathcal{C},E,\kappa )$ is an equivalence for any $\lambda $-topos $\mathcal{E}$.
        \item TFAE:
        \begin{enumerate}
            \item $Sh(\widetilde{\mathcal{C}},\langle \varphi [E]\rangle _{\lambda})$ has enough $\lambda $-points.
            \item  In the diagram below $Lan_{ay}Ran_{\varphi }(ev)$ is an equivalence:
\[\begin{tikzcd}
	{\mathcal{C}} &&& {\mathbf{Set}^{Mod(\mathcal{C})_{<\lambda }}} \\
	{\widetilde{\mathcal{C}}} \\
	\\
	{Sh(\widetilde{\mathcal{C}},\langle \varphi [E]\rangle _{\lambda})}
	\arrow["\varphi"', from=1-1, to=2-1]
	\arrow["ay"', from=2-1, to=4-1]
	\arrow[""{name=0, anchor=center, inner sep=0}, "ev", from=1-1, to=1-4]
	\arrow[""{name=1, anchor=center, inner sep=0}, "{Lan_{ay}Ran_{\varphi } (ev)}"', from=4-1, to=1-4]
	\arrow[""{name=2, anchor=center, inner sep=0}, "{Ran_{\varphi }(ev)}"{description}, from=2-1, to=1-4]
	\arrow["\cong"', draw=none, from=0, to=2]
	\arrow["{\cong }"', draw=none, from=2, to=1]
\end{tikzcd}\]
        \end{enumerate}
        
    \end{itemize}
\end{theorem}

\begin{proof}
    The first claim follows from Proposition \ref{freelambdalim}. All we need is that $\kappa $-filtered ${<\lambda }$ colimits of representables in $\mathbf{Lex}_{\kappa }(\mathcal{C},\mathbf{Set})$ (equivalently in $\mathbf{Set}^{\mathcal{C}}$) are precisely those functors which have pointwise cardinality $<\lambda $. $"\subseteq "$ follows as $\mathcal{C}(x,-)$ has pointwise size $ <\lambda $ and a $<\lambda $ colimit of pointwise $<\lambda $ functors is a quotient of the pointwise disjoint union whose cardinality is $<\lambda $ as $\lambda $ was regular. $"\supseteq "$ follows as (again by regularity) if $M$ is pointwise $<\lambda $ then $|\int M|<\lambda $. The equivalence $\varphi ^{\circ }:\mathbf{Lex}_{\lambda }(\widetilde{\mathcal{C}},\mathcal{E})\to \mathbf{Lex}_{\kappa }(\mathcal{C},\mathcal{E})$ restricts to an equivalence between $\varphi [E]$-, resp.~$E$-preserving functors.
    
    (Note also that $\varphi $ is a morphism of $\kappa $-sites, as since $\varphi $ is $\kappa $-lex the $\langle E\rangle _{\kappa }$-families are taken to $\langle \varphi [E] \rangle _{\kappa }\subseteq \langle \varphi [E] \rangle _{\lambda }$-families.)

    In the second claim the implication $2\Rightarrow 1$ is easy: $ev_M$ for $M\in Mod(\mathcal{C})_{<\lambda }$ is a jointly conservative set of $\lambda $-points. To prove the converse we have to check that $\varphi :(\mathcal{C},E)\to (\widetilde{\mathcal{C}},\varphi [E])$ satisfies Assumption 5.~in Lemma \ref{keylemma}. That is, $Sh(\widetilde{\mathcal{C}},\langle \varphi[E]\rangle _{\lambda})$ has enough $\lambda $-points whose restriction to $\mathcal{C}$ is pointwise $<\lambda $. (This is enough, as if there is an equivalence $\mathbf{Set}^{Mod(\mathcal{C})_{<\lambda }}\to Sh(\widetilde{\mathcal{C}})$ making the triangle commute up to isomorphism, then the same holds for its quasi-inverse, in which case it must be $Lan_{ay}Ran_{\varphi }(ev)$.)
    
    Now pick an arrow $f:x\to y$ in $Sh(\widetilde{\mathcal{C}})$ which is not an isomorphism and choose $\widetilde{N}^*:Sh(\widetilde{\mathcal{C}})\to \mathbf{Set}$ which keeps it non-iso. By Theorem \ref{LS} we can write $\widetilde{N}^* \circ ay \circ \varphi $ as the $\lambda $-filtered colimit of $M_i$'s, each being pointwise $<\lambda $. Since $\varphi ^{\circ }$ was an equivalence (whose quasi-inverse hence preserves this colimit), we can write $\widetilde{N}^*\circ ay$ as the $\lambda $-filtered colimit of $\widetilde{M_i}$'s with $\widetilde{M_i}\circ \varphi \cong M_i$. This colimit is pointwise as $Mod(\widetilde{\mathcal{C}},\varphi [E],\lambda )$ is closed under $\lambda $-filtered colimits in $\mathbf{Set}^{\widetilde{\mathcal{C}}}$. Now the quasi-inverse of $(ay)^{\circ }$ maps this to a $\lambda $-filtered colimit of $\widetilde{M_i}^*$'s with top $\widetilde{N}^*$ such that $\widetilde{M_i}^* \circ ay \circ \varphi \cong M_i$. This colimit is also pointwise: in fact, the pointwise colimit $\widehat{N}^*$ is $\lambda $-lex cocontinuous, so there's an induced map $\widetilde{N}^*\Rightarrow \widehat{N}^*$, inverted by the equivalence $(ay)^{\circ }$. It follows that $\widetilde{M_i}^*(f)$ is non-iso for some $i$. 
\end{proof}

\begin{remark}
\label{rem48}
    The fact that $-\circ \varphi :Mod_{\mathcal{E}}(\widetilde{\mathcal{C}},\varphi [E],\lambda )\to Mod_{\mathcal{E}}(\mathcal{C},E,\kappa )$ is an equivalence for any $\lambda $-topos $\mathcal{E}$ is the same as $-\circ \varphi ^* :Fun^*_{\lambda }(Sh(\widetilde{\mathcal{C}},\langle \varphi [E] \rangle _{\lambda }),\mathcal{E}) \to Fun^*_{\kappa }(Sh(\mathcal{C},\langle E\rangle _{\kappa }),\mathcal{E}) $ being an equivalence. This universal property guarantees that the classifying $\lambda $-topos $Sh(\widetilde{\mathcal{C}},\langle \varphi [E] \rangle _{\lambda })$ depends only on the $\kappa $-topos $Sh(\mathcal{C},\langle E\rangle _{\kappa })$ and not on the $\kappa $-site $(\mathcal{C},E)$. 
    
    Moreover we get a 2-functor $\widetilde{(\ )}$ from "$\kappa $-toposes which have a $<\lambda $ defining $\kappa $-site, $\kappa $-lex cocontinuous functors and natural transformations" to "$\lambda $-toposes, $\lambda $-lex  cocontinuous functors and natural transformations". It takes 1-cells and 2-cells to their unique extensions:
\[\begin{tikzcd}
	{\mathcal{E}} &&& {\widetilde{\mathcal{E}}} \\
	\\
	{\mathcal{F}} &&& {\widetilde{\mathcal{F}}}
	\arrow[""{name=0, anchor=center, inner sep=0}, "{\varphi _{\mathcal{E}}}", from=1-1, to=1-4]
	\arrow[""{name=1, anchor=center, inner sep=0}, "{F^*}"{description}, curve={height=18pt}, from=1-1, to=3-1]
	\arrow[""{name=2, anchor=center, inner sep=0}, "{G^*}"{description}, curve={height=-18pt}, from=1-1, to=3-1]
	\arrow[""{name=3, anchor=center, inner sep=0}, "{\widetilde{F^*}}"{description}, curve={height=18pt}, from=1-4, to=3-4]
	\arrow[""{name=4, anchor=center, inner sep=0}, "{\widetilde{G^*}}"{description}, curve={height=-18pt}, from=1-4, to=3-4]
	\arrow[""{name=5, anchor=center, inner sep=0}, "{\varphi _{\mathcal{F}}}"', from=3-1, to=3-4]
	\arrow["{\alpha }", shorten <=11pt, shorten >=11pt, Rightarrow, from=1, to=2]
	\arrow["\cong"{description}, curve={height=-12pt}, shorten <=9pt, shorten >=9pt, Rightarrow, from=0, to=5]
	\arrow["\cong"{description}, curve={height=12pt}, shorten <=9pt, shorten >=9pt, Rightarrow, from=0, to=5]
	\arrow["{\widetilde{\alpha }}", shorten <=11pt, shorten >=11pt, Rightarrow, from=3, to=4]
\end{tikzcd}\]
(meaning: for each 1-cell $F^*$ we fix a choice of $\widetilde{F^*}$ and an isomorphism in the square, then for each 2-cell $\alpha $ there is only one $\widetilde{\alpha }$ which makes the above 3-cell commute).

It is natural to ask about the properties of this 2-functor. E.g.~if $F$ is an embedding/surjection/open/\dots \ does the same apply to $\widetilde{F}$?
\end{remark}

\section{A $\lambda $-separable $\lambda $-topos has enough $\lambda $-points}

\begin{lemma}\label{compltopos}
    Let $(\mathcal{C},E)$ be a $\lambda $-site with $|E|\leq \lambda $, $\mathcal{C}$ being of local size $\leq \lambda $, $\lambda =cf(\lambda )$, and take a family $(f_j:u_j\to x_0)_j$. If for every $\lambda $-lex $E$-preserving $M:\mathcal{C}\to \mathbf{Set}$ functor the $M$-image of the family is jointly surjective then the $ay$-image is extremal epimorphic. 
\end{lemma}

\begin{proof}
    We repeat the proof of the completeness theorem from the first section. Write $E^{-}$ for the subset of $E$ consisting of non-empty covers (containing at least $1\to 1$). By $T_{x_i}$ we denote a list containing all diagrams consisting of some $E^{-}$-family and a map from $x_i$ to the common codomain. As there are $\geq 1$, $\leq \lambda $ such diagrams we can assume that the list has size $\lambda $.

    Now start filling a $\lambda \times \lambda $-big table with the tasks we have to solve. We shall use the canonical well-ordering $h:\lambda \times \lambda \to \lambda $ with the property $h(\alpha ,\beta )\geq \beta $. In the $0^{th}$ column fill in $T_{x_0}$. Then solve task number $h^{-1}(0)$, whose second coordinate is $\leq 0$ so it is defined. By solving we mean: take the pullback of the given family along the given arrow. So now we see an $(E^{-})^{pb}$-covering family on $x_0$.

    Inductively we build a tree of height $\lambda $, where all branches are cofinal. In a successor step $(\alpha +1)$, where the $\alpha ^{\text{th}}$ object of the branch is $x_{0,i_1,\dots }$ take the table whose columns are filled in with tasks concerning the preceding objects of the branch, fill $T_{x_{0,i_1,\dots }}$ to the $\alpha ^{\text{th}}$-column, and solve task number $h^{-1}(\alpha )$, which is defined. Solving means: this is a diagram formed by an arrow from a preceding object $x'$ above $x_{0,i_1,\dots }$ to the codomain of some $E^{-}$-family, precompose with the $x_{0,i_1,\dots }\to x'$ arrow of the branch, then take the pullback of the family along this composite. In limit steps take the limits (cocompositions) of the branches.

    As we deleted the empty covers from $E$, no branch dies out and we get a locally $(E^{-})^{pb}$-covering continuous cotree of height $\lambda $, and when we take the colimit of the representables $\mathcal{C}(x_0,-)\to \mathcal{C}(x_{0,i_1},-)\to \dots $ along a (cofinal) branch we get a $\mathcal{C}\to \mathbf{Set}$ $\lambda $-lex functor which maps each $E^{-}$-family to a jointly surjective one. 

    If a branch is not preserving an empty cover $\{z\ \ \emptyset \}$, it means that from some element $x_{\vec{i}}$ of the branch there exists a morphism $x_{\vec{i}}\to z$. If the branch is $E$-preserving then by assumption the colimit takes $(f_j:u_j\to x_0)_j$ to a jointly surjective family, in particular $[1_{x_0}]$ is hit, meaning that for some $x_{\vec{i}}$ in the branch the map $x_{\vec{i}}\to x$ factors through some $f_j$. Cut down each branch at such a point, then we see an $\langle E^{-}\rangle _{\lambda }$-family on $x_0$, such that each leg either factors through some $f_j$ or the domain admits a map to some object over which $\emptyset $ is a cover. $ay$ turns these objects initial, so now each leg factors through some $ay(f_j)$, the inductively built family is mapped to an (extremal) epimorphic one, consequently $(ay(f_j):ay(u_j)\to ay(x_0))_j$ is epimorphic. 
\end{proof}

By \cite[Lemma 6.1.4]{makkai} this is sufficient. For the reader's convenience we repeat the proof.

\begin{lemma}
    Let $Sh(\mathcal{C},\langle E\rangle _{\omega })$ be any Grothendieck-topos. Take a natural transformation between sheafified representables $\alpha :a\mathcal{C}(-,x)\Rightarrow a\mathcal{C}(-,x')$. Then there is a cover $(f_i:u_i\to x)_i$ such that each composite $a\mathcal{C}(-,u_i)\xRightarrow{a((f_i)_{\circ })} a\mathcal{C}(-,x)\xRightarrow{\alpha } a\mathcal{C}(-,x')$ is a sheafified post-composition $a((g_i)_{\circ })$ for some $g_i:u_i\to x'$.
\end{lemma}

\begin{proof}
    A consequence of the $+$-construction is that given an arbitrary presheaf $F:\mathcal{C}^{op}\to \mathbf{Set}$ and a natural transformation $\beta :\mathcal{C}(-,x)\Rightarrow F^+$, there is a cover $(f_i: u_i\to x)_i$ such that each $\mathcal{C}(-,u_i)\xRightarrow{(f_i)_{\circ }}\mathcal{C}(-,x)\xRightarrow{\beta } F^+$ factors through the canonical map $F\Rightarrow F^+$. If we start with $F^{++}=aF$ instead of $F^+$, the same remains true as covering trees of height 2 compose to a cover.
    
    Apply this for the composite $\mathcal{C}(-,x)\Rightarrow a\mathcal{C}(-,x)\xRightarrow{\alpha } a\mathcal{C}(-,x')$. We get a cover $(f_i:u_i\to x)_i$ such that for each $i$ there's a commutative diagram
\[\begin{tikzcd}
	{\mathcal{C}(-,u_i)} & {\mathcal{C}(-,x)} & {a\mathcal{C}(-,x)} & {a\mathcal{C}(-,x')} \\
	&&& {\mathcal{C}(-,x')}
	\arrow["{(f_i)_{\circ }}", Rightarrow, from=1-1, to=1-2]
	\arrow[Rightarrow, from=1-2, to=1-3]
	\arrow["\alpha", Rightarrow, from=1-3, to=1-4]
	\arrow[Rightarrow, from=2-4, to=1-4]
	\arrow["{(g_i)_{\circ }}"', Rightarrow, from=1-1, to=2-4]
\end{tikzcd}\]
    whose image under sheafification is exactly what we were looking for.
\end{proof}

\begin{corollary}
    Let $Sh(\mathcal{C},\langle E\rangle _{\omega })$ be any Grothendieck-topos and take a mono\-morphism $\iota :F\Rightarrow a\mathcal{C}(-,x)$. Then there is a cover of $F$ with sheafified representables $a\mathcal{C}(-,u_i)\xRightarrow{\beta _i} F$ such that each $\iota \circ \beta _i$ is a sheafified post-composition.
\end{corollary}

\begin{proof}
    Take any cover of $F$ with sheafified representables $a\mathcal{C}(-,v_i)\Rightarrow F$ then apply the previous lemma for the composites $a\mathcal{C}(-,v_i)\Rightarrow a\mathcal{C}(-,x)$.
\end{proof}

\begin{theorem}
\label{main1}
    Let $(\mathcal{C},E)$ be a $\lambda $-site with $|E|\leq \lambda $, $\mathcal{C}$ being of local size $\leq \lambda $, $\lambda =cf(\lambda )$. Then $Sh(\mathcal{C},\langle E \rangle _{\lambda })$ has enough $\lambda $-points.
\end{theorem}

\begin{proof}
    We have to prove that for an arbitrary proper mono $\iota :F\Rightarrow G$ there is some $M^*:Sh(\mathcal{C},\langle E \rangle _{\lambda })\to \mathbf{Set}$ $\lambda $-lex cocontinuous functor which keeps it proper. It suffices to prove this for $G=a\mathcal{C}(-,x)$, as we can find a cover $(a\mathcal{C}(-,x_j)\Rightarrow G)_j$, then form the pullback:

\[\begin{tikzcd}
	G && {\bigsqcup _j a\mathcal{C}(-,x_j)} \\
	& pb \\
	F && {\bigsqcup _j F_j}
	\arrow["\iota", Rightarrow, from=3-1, to=1-1]
	\arrow[Rightarrow, from=1-3, to=1-1]
	\arrow[Rightarrow, from=3-3, to=3-1]
	\arrow["{\bigsqcup _j \iota _j}"', Rightarrow, from=3-3, to=1-3]
\end{tikzcd}\]
    finally note that $\iota $ is proper iff at least one $\iota _j$ is proper.

    By the previous corollary we have
\[\begin{tikzcd}
	{a\mathcal{C}(-,u_i)} \\
	\dots && F && {a\mathcal{C}(-,x)}
	\arrow["\iota"', Rightarrow, from=2-3, to=2-5]
	\arrow["{\beta _i}"', Rightarrow, from=1-1, to=2-3]
	\arrow["{a((g_i)_{\circ })}", curve={height=-6pt}, Rightarrow, from=1-1, to=2-5]
\end{tikzcd}\]
    and since $\iota $ is proper $(ay(g_i))_i$ cannot be epimorphic. Therefore by Lemma \ref{compltopos} there is some $M:\mathcal{C}\to \mathbf{Set}$ $\lambda $-lex $E$-preserving for which $(M(g_i))_i$ is not jointly surjective. By Theorem \ref{classiftopos} its left Kan-extension $M^*=Lan_{ay}M$ is $\lambda $-lex cocontinuous and $M^* \circ ay \cong M$. But since $M^*(\beta _i)$'s form a jointly surjective family and $M^*(ay(g_i))$'s do not, it follows that $M^*(\iota )$ is proper.
\end{proof}

\begin{theorem}
\label{main2}
We have the following assumptions:
    \begin{itemize}
        \item $\aleph _0\leq \kappa =cf(\kappa )\vartriangleleft \lambda =cf(\lambda )$, $2^{ <\lambda } = \lambda $
        \item $(\mathcal{C},E)$ is a $\kappa $-site, $|\mathcal{C}|,|E| <\lambda  $
    \end{itemize}
Then using the notation of Theorem \ref{ptsthenpresheaf}, in the diagram 
\[\begin{tikzcd}
	{\mathcal{C}} &&& {\mathbf{Set}^{Mod(\mathcal{C})_{<\lambda }}} \\
	{\widetilde{\mathcal{C}}} \\
	\\
	{Sh(\widetilde{\mathcal{C}},\langle \varphi [E] \rangle _{\lambda})}
	\arrow["\varphi"', from=1-1, to=2-1]
	\arrow["ay"', from=2-1, to=4-1]
	\arrow[""{name=0, anchor=center, inner sep=0}, "ev", from=1-1, to=1-4]
	\arrow[""{name=1, anchor=center, inner sep=0}, "{Lan_{ay}Ran_{\varphi } (ev)}"', from=4-1, to=1-4]
	\arrow[""{name=2, anchor=center, inner sep=0}, "{Ran_{\varphi }(ev)}"{description}, from=2-1, to=1-4]
	\arrow["\cong"', draw=none, from=0, to=2]
	\arrow["{\cong }"', draw=none, from=2, to=1]
\end{tikzcd}\]
$Lan_{ay}Ran_{\varphi }(ev)$ is an equivalence.
\end{theorem}

\begin{proof}
    All we have to check is that $\mathbf{Lex}_{\kappa }(\mathcal{C},\mathbf{Set})_{<\lambda }$ has local size $\leq \lambda $. But there are at most as many $F\Rightarrow G$ natural transformations as $\bigsqcup _{x\in \mathcal{C}}F(x) \to \bigsqcup _{x\in \mathcal{C}}G(x)$ maps, and since $\lambda $ is regular this is $(\mu _2)^{\mu _1}$ for some $\mu _1,\mu _2 <\lambda $. This is $\leq \lambda $ by the assumption $\mu <\lambda \Rightarrow 2^{\mu }\leq \lambda $.
\end{proof}

\begin{remark}
    To ensure that $\widetilde{\mathcal{C}}$ has local size $\leq \lambda $ we assumed $2^{<\lambda }=\lambda $. To get that $\widetilde{\mathcal{C}}$ has size $\leq \lambda $ we would need $\lambda ^{<\lambda }=\lambda $. However, these two assumptions are equivalent for a regular cardinal $\lambda $, see \cite[Exercise 5.21]{jech}.
\end{remark}

\begin{remark}
\label{rem57}
    Fix $\kappa =cf(\kappa )\geq \aleph _0$.  Then TFAE:
    \begin{enumerate}
        \item there are arbitrarily large cardinals $\lambda $ satisfying $\kappa \vartriangleleft \lambda =\lambda ^{<\lambda }$
        \item there are arbitrarily large cardinals $\lambda $ satisfying $\lambda ^{<\lambda }=\lambda $ and for $\alpha <\lambda $: $\alpha ^{<\kappa } <\lambda $.
    \end{enumerate}
    This follows from \cite[Fact 2.5]{LIEBERMAN20194560}. These equivalent conditions are true under GCH, or more generally under SCH + "CH holds at arbitrarily large cardinals of cofinality $\geq \kappa $". Indeed, if $\mu > 2^{< \kappa }$, $cf(\mu )\geq \kappa $ and $2^{\mu }=\mu ^{+}$ then $\lambda =\mu ^+$ satisfies both $\lambda ^{<\lambda }=\lambda $ (since $(\mu ^+)^{\mu }=(2^{\mu })^{\mu }=2^{\mu }=\mu ^+$) and $(<\lambda )^{<\kappa } =(< \lambda )$, i.e. $\mu ^{<\kappa }=\mu $ by \cite[Theorem 5.22.ii).b]{jech}.

    So under an additional set theoretic assumption we managed to prove that "every theory is eventually of presheaf type". It is unknown to us whether this follows from ZFC.
\end{remark}

We get the following corollaries:

    \begin{theorem}
    Assume 
    \begin{itemize}
        \item $\aleph _0\leq \kappa =cf(\kappa )\vartriangleleft \lambda =\lambda ^{<\lambda } $
        \item $(\mathcal{C},E)$ is a $\kappa $-site, $|\mathcal{C}|,|E| <\lambda  $
    \end{itemize}

    Then in $\mathbf{Set}^{Mod(\mathcal{C})_{<\lambda }}$ the closure of $\{ev_x :x\in \mathcal{C} \}$ under $\kappa $-cofiltered limits of size $<\lambda $ form a generating set.
\end{theorem}

\begin{theorem}
    Assume:
    \begin{itemize}
        \item $\aleph _0\leq \kappa =cf(\kappa )\vartriangleleft \lambda =\lambda ^{ <\lambda} $
        \item $(\mathcal{C},E)$ is a $\kappa $-site, $|\mathcal{C}|,|E| <\lambda  $
    \end{itemize}

    Take a full subcategory $\mathcal{A}$ of $\mathbf{Lex}_{\kappa }(\mathcal{C},\mathbf{Set})_{<\lambda }$ containing $Mod(\mathcal{C})_{<\lambda }$ and write $I:Mod(\mathcal{C})_{<\lambda }\xhookrightarrow{J_1}\mathcal{A}\xhookrightarrow{J_2} \mathbf{Lex}_{\kappa }(\mathcal{C},\mathbf{Set})_{<\lambda }$ for the full embeddings. Let $\mathcal{E}$ be a $\lambda $-topos and $\mathbf{Set}^{\mathcal{A}}\xrightarrow{M}\mathcal{E}$ be a $\lambda $-lex cocontinuous functor such that for any $(u_i\to x)_i$ family in $E$, the family $M((\restr{ev_{u_i}}{\mathcal{A}}\to \restr{ev_{x}}{\mathcal{A}})_i)$ is extremal epimorphic. Then a $\lambda $-lex cocontinuous extension 

\[\begin{tikzcd}
	{\mathbf{Set}^{\mathcal{A}}} && {\mathbf{Set}^{Mod(\mathcal{C})_{<\lambda }}} \\
	\\
	{\mathcal{E}}
	\arrow["M"', from=1-1, to=3-1]
	\arrow[""{name=0, anchor=center, inner sep=0}, "{J_1^{\circ }}", from=1-1, to=1-3]
	\arrow[""{name=1, anchor=center, inner sep=0}, "{\widetilde{M}}", dashed, from=1-3, to=3-1]
	\arrow["\cong"', draw=none, from=0, to=1]
\end{tikzcd}\]
    exists.
\end{theorem}

\begin{proof}
    Unwinding everything that we've been hiding yields
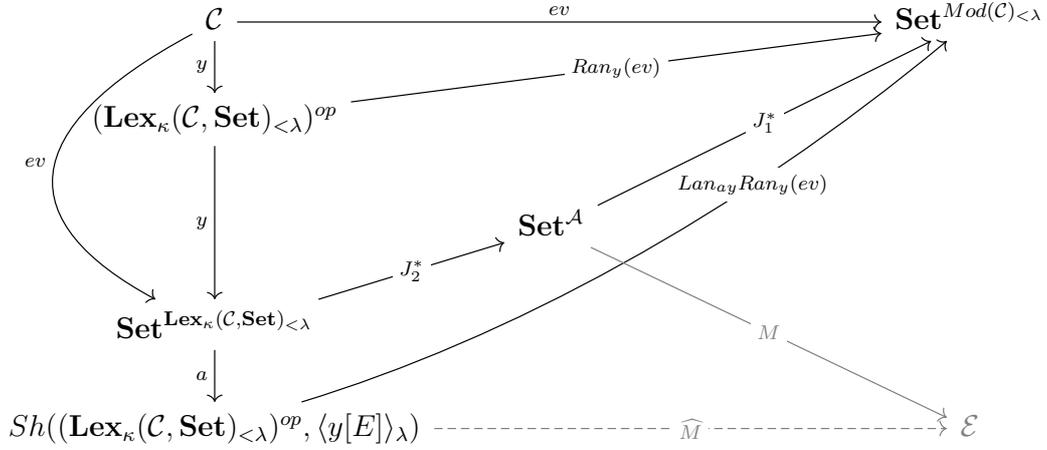
\[
\adjustbox{scale=1}{
\begin{tikzcd}
	{\mathcal{C}} &&&&& {\mathbf{Set}^{Mod(\mathcal{C})_{<\lambda }}} \\
	{(\mathbf{Lex}_{\kappa }(\mathcal{C},\mathbf{Set})_{<\lambda })^{op}} \\
	& {\mathbf{Set}^{\mathcal{A}}} \\
	{\mathbf{Set}^{\mathbf{Lex}_{\kappa }(\mathcal{C},\mathbf{Set})_{<\lambda}}} \\
	{Sh((\mathbf{Lex}_{\kappa }(\mathcal{C},\mathbf{Set})_{<\lambda })^{op},\langle y[E]\rangle _{\lambda})} &&&&& \textcolor{rgb,255:red,128;green,128;blue,128}{\mathcal{E}}
	\arrow["y"', from=1-1, to=2-1]
	\arrow["ev", from=1-1, to=1-6]
	\arrow["{Lan_{ay}Ran_{y } (ev)}"{description, pos=0.7}, curve={height=24pt}, from=5-1, to=1-6]
	\arrow["{Ran_{y }(ev)}"{description}, from=2-1, to=1-6]
	\arrow["y"', from=2-1, to=4-1]
	\arrow["a"', from=4-1, to=5-1]
	\arrow["{J_2^{\circ }}"{description}, from=4-1, to=3-2]
	\arrow["{J_1^{\circ }}"{description}, from=3-2, to=1-6]
	\arrow["ev"', curve={height=80pt}, from=1-1, to=4-1]
	\arrow["M"{description}, color={rgb,255:red,128;green,128;blue,128}, from=3-2, to=5-6]
	\arrow["{\widehat{M}}"{description}, color={rgb,255:red,128;green,128;blue,128}, dashed, from=5-1, to=5-6]
\end{tikzcd}
}
\]

First note that all black triangles commute. Indeed, $I^{\circ }\circ y$ is a $\lambda $-lex functor whose $y$-restriction is isomorphic to $ev$, therefore $I^{\circ }\circ y \cong Ran_y(ev)$. Also $(Lan_{ay}Ran_y(ev)\circ a)\circ y$ is isomorphic to $Ran_y(ev)$, and since both $I^{\circ }$ and $Lan_{ay}Ran_y(ev)\circ a$ are cocontinuous this extends to an isomorphism between  $Lan_{ay}Ran_y(ev)\circ a $ and $ I^{\circ }$.

By assumption $M\circ J_2^{\circ }\circ y$ is $\lambda $-lex $y[E]$-preserving, hence there is a $\lambda $-lex cocontinuous $\widehat{M}$ with $\widehat{M}\circ a\circ y \cong M\circ J_2^{\circ }\circ y$. As $\widehat{M}\circ a$ and $M\circ J_2^{\circ }$ are cocontinuous it follows that $\widehat{M}\circ a \cong M \circ J_2^{\circ }$.

So now if we write $\widetilde{M}=\widehat{M}\circ (Lan_{ay}Ran_y(ev))^{-1}$ we have $\widetilde{M}\circ J_1^{\circ } \circ J_2^{\circ } \cong M\circ J_2^{\circ }$ and therefore $\widetilde{M}\circ J_1^{\circ } \cong M$ since $J_2^{\circ } \circ Lan_{J_2} \cong 1_{\mathbf{Set}^{\mathcal{A}}}$.

\end{proof}

\printbibliography

\end{document}